%% file: main.tex
\documentclass[review,hidelinks,onefignum,onetabnum]{siamart251216}
\usepackage{caption}
\input{ex_shared}
\usepackage{float}
\usepackage{import}
\usepackage{amsmath}

\usepackage{amssymb}
\usepackage{graphicx}
\usepackage{epstopdf}
\usepackage{color}
\usepackage[dvipsnames]{xcolor}
\usepackage{enumitem}
\usepackage{subcaption}

\renewcommand{\i}{{\mathrm i}}
\newcommand{\e}{{\mathrm e}}
\newcommand{\real}{\operatorname{Re}}
\newcommand{\imag}{\operatorname{Im}}
\nolinenumbers

\newcommand{\kC}{\mathcal C}

\newcommand{\kM}{\mathcal M}
\newcommand{\kO}{\mathcal O}

\newcommand{\secref}[1]{Section~\ref{sec:#1}}
\newcommand{\seclab}[1]{\label{sec:#1}}
\newcommand{\eqlab}[1]{\label{eq:#1}}
\renewcommand{\eqref}[1]{(\ref{eq:#1})}

\newcommand{\appref}[1]{Appendix~\ref{app:#1}}

\newcommand{\applab}[1]{\label{app:#1}}
\renewcommand{\S}{\mathbb S}

\newcommand{\N}{\mathbb N}
\newcommand{\C}{\mathbb C}

\newcommand{\R}{\mathbb R}
\newcommand{\Z}{\mathbb Z}

\newcommand{\norm}[1]{\left\lVert#1\right\rVert}

\ifpdf
\hypersetup{
  pdftitle={Slow Passage Through a Hopf},
  pdfauthor={J. M. Christensen, And K. U. Kristiansen}
}
\fi

\begin{document}


%
\maketitle
\begin{abstract}
    In this paper, we consider the slow passage through a Hopf in $\mathbb R^3$ in a certain class of real-analytic systems. Classically, the (simplified) Shishkova problem: $\epsilon \frac{dz}{d\mu} = \lambda(\mu) z+\mathcal O(\epsilon)$, $\lambda(\mu) = \mu-\i$, has been used to illustrate the delay associated with slow passage through a Hopf at $\mu=0$. In this paper, we consider a general $\lambda(\mu)$ with $\lambda(0)\in i\mathbb R\setminus \{0\}$, $\real[\lambda'(0)]\ne 0$, and add an equation for the complex conjugate. It is a basic fact that these systems are local normal forms for any real-analytic $(1,2)$ slow-fast system in $\mathbb R^3$ experiencing slow passage through a Hopf. In this paper, we study these systems under additional (global) assumptions on $\lambda$. In particular, we suppose that $\lambda$ has a simple zero away from the real axis. The remaining assumptions then relate to invariant manifolds of the so-called elliptic system $\dot \mu = -\i\overline{\lambda(\mu)}$ as well as properties of the level set $\operatorname{Re}[\i{\lambda(\mu)}{\lambda(\overline \mu)}]=0$. Under these assumptions, we then provide an asymptotic formula for the exponentially small splitting of attracting and repelling slow manifolds. As a corollary, we also obtain an asymptotic formula for the maximal delay. Our approach is geometric and is inspired by the work of Hayes et al (2016) (using blowup) and Neishtadt (1987,1988) (using so-called elliptic paths) but also by recent work of the second author on exponentially small splitting in unfoldings of the zero-Hopf bifurcation.
\end{abstract}

\section{Introduction}
In this paper, we consider real-analytic systems of the form 
\begin{equation}
\begin{cases}\eqlab{ReS1}
      x' &= \alpha(\mu) x + \beta(\mu) y+\epsilon R( x,y,\mu,\epsilon),\\
 y' &= -\beta(\mu) x + \alpha(\mu) y+\epsilon S(x, y,\mu,\epsilon),\\
 \mu' &=\epsilon,
\end{cases}
\end{equation}
where
\begin{equation}\eqlab{HopfC}
    \alpha(0)=0, \quad \alpha'(0)\ne 0, \quad \beta(0)\ne 0,
\end{equation} 
and $0<\epsilon\ll1$.
It is without loss of generality to assume that $\alpha'(0)>0$ and $\mu\alpha(\mu)>0$ for all $\mu$ in a neighborhood of 0. Under these assumptions, $(x,y)=(0,0)$ defines an attracting (repelling) critical manifold of \eqref{ReS1} with $\epsilon = 0$ for $\mu<0$ ($\mu>0$, respectively), with eigenvalues of the linearization being $\alpha(\mu)\pm i \beta(\mu)$. At $\mu=0$, the critical manifold loses normal hyperbolicity due to a pair of imaginary eigenvalues. For $0<\epsilon\ll1$, the system \eqref{ReS1} therefore undergoes a slow passage through a Hopf. In fact, any system in $\R^3$ undergoing a slow passage through a Hopf can locally be brought into the form (1.1), see \cite{neishtadt1987a} and App A for further details. 
 
The problem of the slow passage through a Hopf has a long history, see e.g. \cite{neishtadt1987a,neishtadt1988a}. In these references, Neishtadt showed that there is a delay for general analytic systems in which attracting slow manifolds $\xi^a(\mu,\epsilon) = (x^a(\mu,\epsilon),y^a(\mu,\epsilon))$ with $\mu<0$ pass through $\mu=0$ and leave a neighborhood of the repelling critical manifold for $\mu=\kO(1)>0$ as $\epsilon \rightarrow 0$. The delayed loss of stability is a consequence of the exponentially small splitting of attracting and repelling slow manifolds on either sides of $\mu=0$, see e.g. \cite{de2020a,hayes2015a}.

In this paper, we are concerned with an asymptotic formula for the splitting of the attracting and repelling slow manifolds and formulas for the maximal delay, going beyond the estimates in \cite{neishtadt1987a}.


\subsection{Problem description}
It is well-know that treatments of exponentially small phenomena require extensions into the complex plane, see \cite{lazutkin2005a,neishtadt1987a,neishtadt1988a}. 
We therefore consider \eqref{ReS1} for $\mu\in \Sigma\subset \C$, where $\Sigma$ is a complex strip 
\begin{align}\eqlab{strip}
    \Sigma= [-C_1,C_2] \times \i [-C_3,C_4],\quad C_j>0, \,j\in\{1,2,3,4\},
\end{align} 
with $\i$ denoting the imaginary unit, and introduce the complex variables $z=x+\i y$, $w=x-\i y$ such that
\begin{equation}\eqlab{ComplexI}
\begin{cases}
z' &= \lambda(\mu) z + \epsilon F(z,w,\mu,\epsilon),\\
 w' &= \nu(\mu) w + \epsilon G(z,w,\mu,\epsilon),\\
 \mu' & = \epsilon.
\end{cases}
\end{equation}
Here we have defined
 \begin{align}\eqlab{lambdanudefn}
  \lambda(\mu) := \alpha(\mu)-\i\beta(\mu),\quad \nu(\mu) : = \alpha(\mu)+\i\beta(\mu),
 \end{align}
and
\begin{equation}
\begin{aligned}\eqlab{FG}
 F(z,w,\mu,\epsilon) : = R\left(\frac{z+w}{2},\frac{z-w}{2\i},\mu,\epsilon\right)+\i S\left(\frac{z+w}{2},\frac{z-w}{2\i},\mu,\epsilon\right),\\
 G(z,w,\mu,\epsilon) : = R\left(\frac{z+w}{2},\frac{z-w}{2\i},\mu,\epsilon\right)-\i S\left(\frac{z+w}{2},\frac{z-w}{2\i},\mu,\epsilon\right).
\end{aligned}
\end{equation}
Notice that  (with bars denoting complex conjugation)
\begin{align}\eqlab{symFL}
F(z,w,\mu,\epsilon) = \overline{G(\overline{w},\overline{z},\overline{\mu},\epsilon)},\quad \lambda(\mu)=\overline{\nu(\overline \mu)},
\end{align}
for all $z,w,\mu,\epsilon$. Here we have used that the system \eqref{ReS1} is real-analytic. This leads to the following obvious fact (obtained by conjugating \eqref{ComplexI}):
\begin{lemma}\label{lemmasym}
   $(z,w,\mu)(t)$ is a solution of \eqref{ComplexI} if and only if $(\overline{w},\overline{z},\overline{\mu})(\overline t)$ is. 
\end{lemma}
 
This symmetry relates the positive halfspace $\imag(\mu) \geq 0$ to the negative halfspace $\imag(\mu) \leq 0$ and will be important in the following.
 
Following \cite{krupa_extending_2001,neishtadt1987a,neishtadt1988a}, we will analyze \eqref{ComplexI} in a reparametrized form obtained by multiplication of the right hand side by $-\i\overline{\lambda(\mu)}$: 
\begin{equation}
    \begin{cases}\eqlab{ELI}
    z' &= -\i\vert\lambda(\mu)\vert^2 z  -\i\epsilon \overline{\lambda(\mu)}F(z,w,\mu,\epsilon),\\
    w' &= -\i\overline{\lambda(\mu)}\nu(\mu) w - \i\epsilon \overline{\lambda(\mu)} G(z,w,\mu,\epsilon),\\
     \mu' & = -\i \epsilon \overline{\lambda(\mu)}.
\end{cases}
\end{equation}
Crucially, the $\mu$-equation decouples. It takes the following form with respect to the slow time:
\begin{align}
\dot{\mu} = - \i\overline{\lambda (\mu)}.\eqlab{THIS}
\end{align}
The trajectories of \eqref{THIS} are called elliptic paths (or Stokes contours, see \cite[Remark 2.2]{hayes2015a}). The desirable property of these paths is that the linear part of the $z$-equation is oscillatory.
 
We consider our system under a set of assumptions:
\begin{enumerate}[label=(A\arabic*)]
    \item \label{A1} $\lambda(\mu)$ has a single root $\mu_*$ in $\Sigma$ which is simple. Moreover, $\real[\lambda(\mu)]\gtrless 0$ for $\mu\gtrless 0$, respectively, in $\Sigma \cap \mathbb R$.
\end{enumerate}
It is without loss of generality to assume that $\imag(\mu_*)>0$. (If not, then we can exchange $z$ and $w$, as well as $\lambda$ and $\nu$.) It follows from a simple calculation that when \ref{A1} holds then \eqref{THIS} has a unique hyperbolic saddle singularity at $\mu_*$ with stable and unstable manifolds. For further details, we refer to Section 2. Our next assumptions relate to these invariant manifolds in the complex $\mu$-plane.
\begin{enumerate}[resume*]
    \item \label{A2} The stable manifold of the point $\mu_*$ intersects the real axis in a unique point $\mu_s \in \Sigma$ with $\mu_s < 0$.
    \item \label{A3} The unstable manifold of the point $\mu_*$ intersects the real axis in a unique point $\mu_u \in \Sigma$ with $\mu_u >0$.
\end{enumerate}
Notice that $\overline{\mu_*}$ is a root of $\nu$ by \eqref{symFL}.

Our work can be seen as an extension of \cite{hayes2015a}. In this reference, the authors studied the (simplified) Shishkova problem 
\begin{equation}\eqlab{Shish}
    \begin{aligned}
        z' &= \lambda(\mu)z+\epsilon h(\mu),\\
        \mu'&=\epsilon,
    \end{aligned}
\end{equation}
with $h$ analytic for $\vert \imag(\mu) \vert \leq 2$, and where \begin{align}\eqlab{shislambda}\lambda(\mu)=\mu-\i.\end{align} (More precisely, they actually consider $\lambda(\mu)=\mu+\i$, but we prefer to have the root in the $\imag(\mu)>0$ halfplane and therefore consider the equivalent version \eqref{shislambda} instead). In particular, the authors of \cite{hayes2015a} presented a novel geometric approach for the description of the exponentially small splitting of the attracting and repelling slow manifolds, defined for $\mu<0$ and $\mu>0$, respectively. Their description is based upon writing \eqref{Shish} in a reparametrised form obtained by a multiplication of the right hand side by $-\overline{\lambda(\mu)}= - (\overline{\mu}+\i)$ such that 
\begin{equation}\eqlab{hayeshyp}
    \begin{aligned}
        z' &= -\vert\lambda(\mu)\vert^2z-\epsilon \overline{\lambda(\mu)}h(\mu),\\
        \mu'&=-\epsilon \overline{\lambda(\mu)}.
    \end{aligned}
\end{equation}
For this system with $\epsilon=0$, the set defined by $z=0$ is hyperbolic whenever $\lambda \neq 0$, and this enables the use of geometric singular perturbation theory (GSPT). Since $\lambda(\i)=0$, the point $(z,\mu) = (0,\i)$  is degenerate for \eqref{hayeshyp} for $\epsilon=0$, and it is therefore treated in \cite{hayes2015a} by a blowup transformation \cite{krupa_extending_2001}. Notice that $\mu$ decouples from \eqref{hayeshyp}, taking the following form 
\begin{equation}\eqlab{THIS2}
    \dot \mu = -\overline{\lambda(\mu)},
\end{equation}
with respect to slow time.
Trajectories of (\theequation) are called hyperbolic paths (or anti-Stokes contours, again see \cite[Remark 2.2]{hayes2015a}).
 
The system \eqref{ComplexI} can be seen as an extension of \eqref{Shish} with an additional equation of the complex conjugate $w$ of $z$. 
Notice that the perturbations in \eqref{ComplexI} are assumed to be fully nonlinear (in contrast to \eqref{Shish}) and we are interested in a local phenomena. Here we use "local" to contrast it to the situation in \cite{hayes2015a} where they fix global slow manifolds at $\mu \to \pm \infty$. 

In Shishkova's original work \cite{MR324157}, they actually studied a slightly more complicated problem than \eqref{Shish}:
\begin{equation}\eqlab{shishkov_general}
\begin{aligned}
z'  &=\lambda(\mu) z +\gamma z^2 w-\epsilon,\\
w' &=\nu(\mu) w +\gamma w^2z-\epsilon,\\
\mu'&=\epsilon,
\end{aligned}
\end{equation}
with $\lambda(\mu) = \mu-\i$, $\nu(\mu)=\mu+\i$, see also \cite[Eqs. (3.3)-(3.4)]{engler2026a} replacing their $(u^*,u)$ with $(z,w)$. It is possible to bring \eqref{shishkov_general} into \eqref{ComplexI}, 
but only away from $\mu=\pm \i$ (see Section \ref{sec:disc} below for a further discussion). The main result of \cite{MR324157} is that there is a maximal delay, in the sense that any slow manifold $(z,w)=\xi^a(\mu,\epsilon)$ with $\mu\le -1$ leaves a neighborhood of $(z,w)=(0,0)$ around $\mu=1$. The general idea of \cite{MR324157} is to use the solution of the associated linear problem and define an appropriate fixed point formulation for $\xi^a(\mu,\epsilon)$ with $\mu\in \mathbb C$ complex (in an appropriate domain).  Interestingly, the result and approach of  Shishkova was recently revisited in \cite{engler2026a}. This reference simplifies some aspects of the proof and obtain a more refined result. It is open whether Shishkova's approach can be applied to more general systems.

A significant difficulty in extending \eqref{Shish} to \eqref{ComplexI} stems from controlling the growth of the extra variable $w$. This complication was also noted in \cite{KRUPA20102841}, where the authors studied the slow passage through a Hopf in the analysis of the folded saddle node. As a result, it is crucial to use the elliptic paths, recall \eqref{ELI}, rather than the hyperbolic ones. This is supported by \cite{KRUPA20102841} and in line with the seminal work of \cite{neishtadt1987a,neishtadt1988a}. Notice that the exponential growth/decay of $w$ in \eqref{ELI} is determined (to leading order in $\epsilon$) by the quantity 
\begin{align}\eqlab{quantity}
\real[-\i\overline{ \lambda(\mu)}\nu(\mu)] = \real[\i\lambda(\mu){\lambda(\overline \mu)}].
\end{align}
For the expression on the right hand side we have used \eqref{symFL}.
This motivates the definition of the following set  
\begin{equation}\nonumber
    \kC =\{\mu \in \Sigma :\real[-\i\overline{ \lambda(\mu)}\nu(\mu)]=0\},
\end{equation}
which separates exponential growth from exponential decay. 
Due \eqref{HopfC} and assumption \ref{A1}, we have $0 \in \kC$, $\mu_* \in \kC$ and $\overline{\mu_*} \in \kC$. We now make the following assumptions on $\kC$:
\begin{enumerate}[label=(B\arabic*)]
    \item \label{B1} $\kC$ contains a smooth and connected $1$-dimensional manifold (with boundary) $\kC_*\subset \kC$ which contains $\mu=0$ and has $\mu_*$ and $\overline{\mu_*}$ as it is boundary points. In other words, $\kC_*$ is a smooth curve without self-intersections that connects $\mu_*$ with $\overline{\mu_*}$ through $\mu=0$. Furthermore, $\kC_*$ is regular in the sense that the differential of $(a,b)\mapsto \real[-\i\overline{ \lambda(a+\i b)}\nu(a+\i b)]$ is non-zero for any $a+\i b\in \mathcal C_*$. 
    \item \label{B2} $\kC_*$ is non-tangent to the stable and unstable manifolds of $\mu_*$ and there are no contact points away from $\mu_*$ along $\kC_*$, i.e. the flow of the elliptic system \eqref{THIS} intersects the interior of $\kC_*$ transversely.
\end{enumerate}
The assumptions are illustrated in Figure \ref{EBox}. Notice that the absence of contact points in \ref{B2} forces $\kC_*$ to be contained within the triangular-like region (shaded in Figure \ref{EBox}) bounded by the two branches of the stable and unstable manifolds. This is a simple consequence of Rolle's theorem. 
 We then note that \ref{A1} and \ref{B1} implies that $\real[-\i\overline{ \lambda(\mu)}\nu(\mu)]$  is negative (positive) on the $\mu$-negative (positive, respectively) side of $\kC_*$. Finally, due to \eqref{symFL}, we have that $\overline{\kC_*}=\kC_*$. We now add our final assumption:
 \begin{enumerate}[resume*]
     \item \label{B3} $\kC_*$ is the only component of $\kC$ in the region bounded by the branches of the stable and unstable manifolds of $\mu_*$ and the real axis. 
 \end{enumerate}

\begin{figure}
    \centering
    \includegraphics[width=0.7\linewidth]{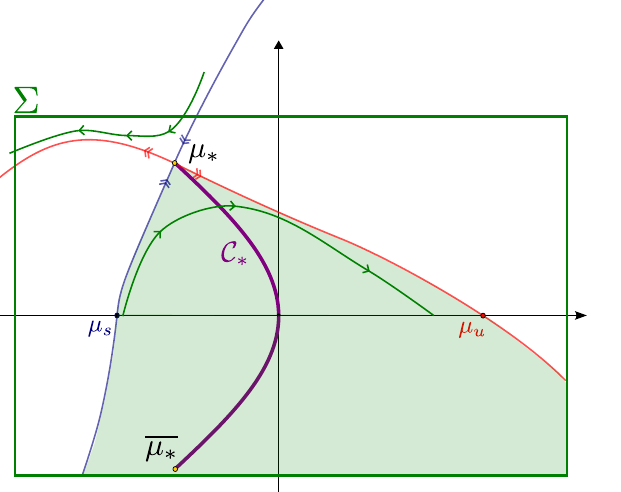}
    \caption{The phase portrait of the elliptic system \eqref{THIS} under the assumptions \ref{A1}-\ref{A3}, \ref{B1}-\ref{B3}. Notice that the point $\mu=\mu_*$ is a hyperbolic saddle and \ref{A2} and \ref{A3} assume that the associated invariant manifolds (blue and red) intersect the real axis in $\mu_s<0$ and $\mu_u>0$, respectively. Moreover, \ref{B1} and \ref{B2} assume that $\mathcal C_*$ (magenta), connecting $\overline \mu_*$ with $\mu_*$ through $\mu=0$, is not tangent to the invariant manifolds at $\mu_*$ and that the elliptic flow of \eqref{THIS} is transverse to $\mathcal C_*$ away from $\mu_*$. Finally, \ref{B3} says that there are no other components of $\kC$ in the shaded region.
    }
    \label{EBox}
\end{figure}

It is important to emphasize that $w$ is only the conjugate of $z$ when $\mu\in \mathbb R$. In particular, if $w(0)=\overline z(0)$ then by uniqueness of solutions we have that $w(\overline t)=\overline{z(t)}$, recall Lemma \ref{lemmasym}. In other words, if $z$ is known then we also know $w$ but on a separate path (when $t\notin \mathbb R$) in the complex $t$-plane. As a consequence, we have to treat $(z,w)$ as distinct variables when extending the slow manifolds into the complex plane (by following the flow of \eqref{ELI}). 
\subsection{Main results}
We now present the main results. For this, we fix any attracting slow manifold $(x,y) = \xi^a(\mu,\epsilon)$ defined for $\mu$ in a real neighborhood of $\mu_s$, $0 \leq \epsilon \ll 1 $. Similarly, we fix any repelling slow manifold  $(x,y) = \xi^r(\mu,\epsilon)$ defined for $\mu$ in a real neighborhood of $\mu_u$, $0 \leq \epsilon \ll 1 $. Here we have used \ref{A1}. Now, it is a basic fact (using \ref{A2}, see Proposition \ref{Extension1} below) that $\xi^a$ and $\xi^r$ both extend smoothly to the nonhyperbolic point $\mu=0$ (through the forward and backward flow, respectively). We can therefore define the difference 
\begin{equation}\nonumber
    \Delta(0,\epsilon) = \xi^a(0,\epsilon) - \xi^r(0,\epsilon), \quad 0 \leq \epsilon \ll 1.
\end{equation}
Our first result gives an asymptotic formula for $\Delta(0,\epsilon)$.
\begin{theorem}\label{TheoremSplitting}
    Consider \eqref{ReS1} and assume \ref{A1}-\ref{A3}, \ref{B1}-\ref{B3}, and that $$F_0 := F(0,0,\mu_*,0) \neq 0,$$ recall \eqref{FG}, is non-zero. Moreover, define 
    \begin{align}\eqlab{lambda0defn}
    \lambda_0 := \lambda'(\mu_*),
    \end{align} and the complex number 
    \begin{equation*}
    \resizebox{0.99\linewidth}{!}{$
        W(\epsilon) := \frac{F_0}{\sqrt{\vert \lambda_0 \vert}}\exp \left({-\i \epsilon^{-1} \imag\left[ \int_0^{\mu_*}\lambda(\mu)d\mu\right] -\int_0^{\mu_*}\partial_z F(0,0,\mu,0) d\mu} -\frac{3\i}{2}\arg(\lambda_0) \right),
        $}
    \end{equation*}
with $\arg(\lambda_0)$ denoting a fixed argument of $\lambda_0\in \C\setminus\{0\}$.
    Then 
    \begin{align}\eqlab{negativequantity}
        \operatorname{Re}\left[\int_0^{\mu_*}{\lambda(\mu)}d\mu\right]>0,
    \end{align}
    and there exists an $\ell \in \{0,1\}$ such that
    the splitting $\Delta(0,\epsilon) = (\Delta_x(0,\epsilon),\Delta_y(0,\epsilon))$ of the slow manifolds $(x,y)=\xi^{a,r}(\mu,\epsilon)$ of \eqref{ReS1} at $\mu=0$ has the following asymptotic expansion 
    \begin{equation}
        \begin{aligned}\eqlab{finalsplitting}
        \Delta_x(0,\epsilon) = &\sqrt{(2\pi \epsilon)}(-1)^{\ell-1} \exp\left({-\epsilon^{-1}\operatorname{Re}\left[\int_0^{\mu_*}{\lambda(\mu)}d\mu\right] }\right)\real\left[ W(\epsilon)(1+\mathcal O(\sqrt{\epsilon}) \right],\\
         \Delta_y(0,\epsilon) = &\sqrt{(2\pi \epsilon)} (-1)^{\ell-1} \exp\left({-\epsilon^{-1}\operatorname{Re}\left[\int_0^{\mu_*}{\lambda(\mu)}d\mu\right] }\right)\imag\left[ W(\epsilon)(1+\mathcal O(\sqrt{\epsilon}) \right].
        \end{aligned}
    \end{equation}
\end{theorem}
 \begin{remark}\label{nonuniqueness}
Notice that:
\begin{equation}
    \vert W(\epsilon) \vert = \left\vert \frac{F_0}{\sqrt{\lambda_0}} \right\vert  \exp \left({-\real\left[\int_0^{\mu_*}\partial_z F(0,0,\mu,0) d\mu \right]} \right), 
\end{equation}
for all $\epsilon>0$. The quantity $\ell$ relates to the angle ($\theta^{\kC}$) of the contour $\kC_*$ at the root $\mu_*$ and is quantified in the proof below, see \secref{secsplit} for details. Moreover, the slow manifolds are not unique, but for the statement it is only important that the attracting and repelling slow manifolds are initially fixed in real neighborhoods of $\mu_s$ and $\mu_u$, respectively, as described above. In contrast, if one instead fixes slow manifolds as graphs over compact intervals contained within $(\mu_s,0)$ and $(0,\mu_u)$, respectively, then their splitting at $\mu=0$ will in general be further apart. We are not concerned with the details of such splittings. 
\end{remark}

The following result gives sufficient conditions for \ref{A1}-\ref{A3} and \ref{B1}-\ref{B3}.
\begin{proposition}\label{mainprop}
    Consider \eqref{ReS1} with 
\begin{align}\eqlab{alphabetalin}
    \begin{cases} \alpha(\mu) = \real(\lambda_0) \mu,\\ \beta(\mu) = -\imag(\lambda_0)\mu + 1,
    \end{cases}
\end{align}
for  $\real(\lambda_0)>0$, 
    so that 
    \begin{align}\eqlab{lambdanulin}
    \begin{cases} \lambda(\mu) = \lambda_0\mu -\i,\\
     \nu(\mu) = \overline{\lambda_0}\mu +\i, 
     \end{cases}
\end{align}
and $$\mu_* = \frac{\i}{\lambda_0}.$$
 Then \ref{A1}-\ref{A3}, \ref{B1}-\ref{B3}  all hold true (with $\mu_s= -\mu_u= -\vert \lambda_0\vert^{-1}$) if and only if 
\begin{align}\eqlab{condlambda12}
    {\real(\lambda_0)}>\sqrt{3}\vert {\imag(\lambda_0)}\vert.
\end{align}

\end{proposition}
Notice that once we have assumed that there is a Hopf at $\mu=0$, then it is without loss of generality (through a rescaling of time) to take $\lambda(0)=-\i$, $\nu(0)=\i$ (as in \eqref{lambdanulin}). We also note that the quantity \eqref{negativequantity} takes the following form:
\begin{align*}
    \operatorname{Re}\left[\int_0^{\mu_*}{\lambda(\mu)}d\mu\right] = \frac{\real(\lambda_0)}{2\vert \lambda_0\vert^2},
\end{align*}
for \eqref{lambdanulin}. The case \eqref{shislambda} studied in \cite{hayes2015a} corresponds to $\lambda_0=1$. It is then easy to see that \eqref{finalsplitting} with $\lambda_0=1$ and $F$ independent of $z$ agrees with \cite[Theorem 2.1]{hayes2015a}.

The conditions \ref{A1}-\ref{A3}, \ref{B1}-\ref{B3}  are robust (structurally stable), and by Proposition \ref{mainprop}, Theorem \ref{TheoremSplitting} therefore applies to \eqref{ReS1} with real-analytic perturbations of \eqref{alphabetalin}  satisfying \eqref{condlambda12}. 

To prove Theorem \ref{TheoremSplitting}, we work in the $(z,w)$-coordinates of \eqref{ComplexI}. In further details, our strategy is to extend the corresponding $\xi^a$ using the elliptic flow of \eqref{ELI} all the way up to $\kC_* \cap \{\imag(\mu) \geq 0\}$; for $\imag(\mu) \leq 0$ we use the symmetry from Lemma \ref{lemmasym}. (For simplicity, we continue to use the same symbol for $\xi^{a,r}$ when expressed in the $(z,w)$-coordinates.) Here we use \ref{A2}, \ref{A3} and \ref{B2}-\ref{B3}. In particular, we use that $\real[-\i\overline{ \lambda(\mu)}\nu(\mu)]<0$ on the $\mu$-negative side of $\kC_*$ in order to control the growth of $w$. Similarly, the corresponding $\xi^r$ is extended to $\kC_*$ using the backward elliptic flow of \eqref{ELI}. Here, we analogously use that $\real[-\i\overline{ \lambda(\mu)}\nu(\mu)]>0$ on the $\mu$-positive side of $\kC_*$.  In this way, we can define 
\begin{equation}\nonumber
    \Delta(\mu,\epsilon) = \xi^a(\mu,\epsilon) - \xi^r(\mu,\epsilon),
\end{equation}
for $\mu$ on a subset of $\kC_*$. In order to extend $\Delta(\mu,\epsilon)$ for $\mu\in \kC_*$ close to $\mu_*$, uniformly with respect to $0<\epsilon\ll 1$, we consider (as in \cite{hayes2015a}) a blowup of the degenerate point $(z,w,\mu,\epsilon) = (0,0,\mu_*,0)$, recall \ref{A1}, of the following form 
\begin{equation}\eqlab{BlowUpIntro}
    \begin{cases}
         z  =r \breve{z},\\
        w = r \breve{w},\\
        \mu = \mu_*+r\breve{\mu},\\
        \epsilon =r^2\breve{\epsilon},
    \end{cases}
\end{equation}
with $r \geq 0, (\breve{z},\breve{w},\breve{\mu},\breve{\epsilon}) \in \S^3\subset \mathbb C^4$. Here $\S^3$ denotes the complex 3 sphere. In further details, we demonstrate that the attracting and repelling slow manifolds are asymptotic to certain special solutions on sectors of the blowup sphere. In this way, we can determine an order one splitting (in the blowup coordinates)  near $\mu_*$ (as $\epsilon\to 0$) under the assumption that $F_0 \neq 0$. Our construction here follows \cite{kriszm1} using a combination of different asymptotic series in charts associated with the blowup. Knowing the splitting near $\mu_*$, we can then derive a linear equation for the difference on $\kC_*$ via the mean value theorem. By going to projective variables and using assumption \ref{B2} (no contact points), this linear equation can be diagonalized globally by Fenichel's theory \cite{fen3,jones_1995}.  Integrating this system gives the final result. This approach for treating the difference follows \cite{kristiansenzerophopf2,kristianZeroHopfArbitrary}.

Now, Theorem \ref{TheoremSplitting} gives the splitting of the invariant slow manifolds at $\mu=0$. We can extend $\xi^a(\mu,\epsilon)$ to $\mu>0$ by application of the forward flow of \eqref{ReS1}. In this way, we can also extend the (real) difference $\Delta(\mu,\epsilon)=(\Delta_x(\mu,\epsilon),\Delta_y(\mu,\epsilon))$ to $\mu>0$ small enough. In the following corollary, we are interested in the point $\mu>0$ where $\Delta(\mu,\epsilon) = \kO(1)$.
\begin{corollary}\label{Delay}
    Consider \eqref{ReS1} and assume \ref{A1}-\ref{A3}, \ref{B1}-\ref{B3} and that \newline $F_0\neq 0$. Fix slow manifolds $(x,y)=\xi^{a,r}(\mu,\epsilon)$ near $\mu=\mu_s$ and $\mu=\mu_u$, respectively, as above and any real number $\delta>0$ small enough. Then for $0 \ll \epsilon < 1$ sufficiently small there exists a unique $\mu_\delta(\epsilon) >0$, given by 
    \begin{equation}\eqlab{Deq0}
        \mu_\delta(\epsilon) = \mu_u - \frac{\epsilon \log(\epsilon)}{2\alpha(\mu_u)}(1+o(1)),
    \end{equation}
    such that
    \begin{equation}\nonumber
        \vert \Delta(\mu_\delta,\epsilon) \vert= \delta.
    \end{equation}
\end{corollary}
The proof is obtained by transporting the difference $\Delta(\mu,\epsilon)$ at $\mu=0$, described in Theorem \ref{TheoremSplitting}, in the positive $\mu$ direction. In further details, we derive an equation for $\Delta(\mu,\epsilon)$ which is dominated by a linear part. The result is then obtained by Gronwall's inequality and the implicit function theorem.
\subsection{Discussion}\label{sec:disc}
The quantity \eqref{Deq0} gives a precise notion of the maximal delay. The point $\mu_u$ is also known as a buffer point \cite{hayes2015a}. For $\mu>0$, our results implies that the splitting in Theorem \ref{TheoremSplitting} is exponentially small for $\mu<\mu_u$ uniformly bounded away from $\mu_u$, rapidly transitions to $\kO(1)$ in a $\epsilon \log(\epsilon)$-neighborhood of $\mu_u$, and finally grows exponentially (unbounded) for $\mu>\mu_u$. In conclusion, $\mu_u$ acts as a buffer point from where solutions cannot stay close to repelling the slow manifolds. A similar results holds under the backwards flow, where $\mu_s$ acts a buffer point in backwards time.
 
In line with Remark \ref{nonuniqueness}, if we were to fix an attracting slow manifold as a graph over a compact interval contained within $(\mu_s,0)$, then there will in general be a shorter delay. This is due to the fact that in this case we cannot (in general) extend the invariant manifold to $\mu_*$ by following the elliptic flow, see Figure \ref{EBox}. This is also described in Neishtadt's work \cite{neishtadt1987a} for general systems in $\mathbb R^n$.

Although we are inspired by Neishtadt's work \cite{neishtadt1987a} and his use of elliptic paths, our results are different. In particular, we obtain explicit, quantitative asymptotic formulas for the splitting of the slow manifolds in a simpler setting in $\mathbb R^3$. This contrasts the qualitative description in \cite{neishtadt1987a} for general systems. 

It is important to emphasize that although \eqref{ReS1} is a local normal form for the slow passage through the Hopf in $\mathbb R^3$ (see \appref{app}), our (global) assumptions  \ref{A2}-\ref{A3}, \ref{B1}-\ref{B3}, are (as already pointed out in \cite{neishtadt1987a}) not the general case. In particular, the construction in \appref{app} breaks down near roots of $\lambda$ and $\nu$. This can also be seen immediately when trying to bring \eqref{shishkov_general} into the form of \eqref{ComplexI}.


Therefore to obtain \eqref{ReS1} from a general slow passage through a Hopf in $\mathbb R^3$ requires additional assumptions (e.g. a globally defined critical manifold and perturbations of $\mathcal O(\epsilon^2)$, see further discussion in Remark \ref{RemA1}, see \appref{app}).

During the time of writing, we were made aware of the following preprint \cite{neishtadt2026maximaldelaystabilityloss} by Neishtadt, which also pursues maximal delay, but in general systems. Here $\mu_*$ is instead assumed to be a fold singularity of the critical manifold in the complex plane. Neishtadt argues that this is the general case. At the same time, we notice that Neishtadt also includes an assumption on the lack of contact points (written in the form of $\lambda(\mu)\ne \nu(\mu))$).
 
In any case, our approaches are different and we also aim to pursue our geometric approach for the generic systems studied in \cite{neishtadt2026maximaldelaystabilityloss} in future work. We are confident that this is possible. 
In fact, our method of proof is also applicable for Shishkova's original system \eqref{shishkov_general}. This can be checked explicitly by following the proof of Theorem \ref{TheoremSplitting} step by step. The splitting in the Shishkova system follows almost identically with some minor modifications to Proposition \ref{Extension1}, \ref{E4}, \ref{Coordinates2} and Lemma \ref{lemma45}. The asymptotic expansion of the splitting is still given by the expressions in Theorem \ref{TheoremSplitting} with:
\begin{equation*}
    W(\epsilon)=-1,\quad \int_{0}^{\mu_*} \lambda(\mu) d\mu = \frac{1}{2}.
\end{equation*}
For further details, we refer to the forthcoming PhD-thesis of the first author.

Finally, we mention that our approach breaks down near contact points. Firstly, if there are contact points, then we cannot use our approach to follow the elliptic flow up to $\kC_*$. Figure \ref{KSystemFig} below illustrates this point, see Section \ref{proofprop} for further details. Moreover, contact points are also an obstacle to our approach to the difference. Indeed, at contact points the equation for the difference loses normal hyperbolicity, see Section \ref{secdiff} for further details. We aim to study contact points in future work.

\subsection{Outline}
The paper is structured as follows: In Section \ref{basics}, we first prove that the elliptic and hyperbolic systems have hyperbolic singularities near $\mu=\mu_*$ under the assumption \ref{A1}. Moreover, we introduce certain angles associated with tangents of the invariant manifolds and $\kC_*$ at $\mu_*$, that will be important for the blowup (using polar coordinates near $\mu_*$). Next in Section \ref{asympbu}, we describe formal invariant manifolds, first uniformly bounded away from $\mu_*$ and then subsequently near $\mu_*$ using blowup and two separate charts.  The formal series in Section \ref{asympbu} are crucial for our analytic extensions of the slow manifolds to $\kC_*$ in Section \ref{secextension}. Here we first extend the fixed attracting slow manifold to the non-hyperbolic point $\mu = 0$ using the flow of \eqref{ComplexI}, see Proposition \ref{Extension1}. Next, we use the extended slow manifold as initial conditions for the flow of \eqref{ELI} for $\mu \in (\mu_s,0)$ bounded uniformly away from $\mu_s$. In this way, the attracting slow manifold is extended to a sub-arc of $\kC_*$ bounded away from $\mu_*$, see Figure \ref{EBox} and Proposition \ref{E4}. Finally, we consider initial conditions close to $\mu_s$ and extend the attracting slow manifold to the blowup of $\mu_*$, see Proposition \ref{F4Res}. Here we show that the attracting slow manifold is determined by a special solution (related to the formal series of Section \ref{asympbu}) on a sector centered around the blowup of $\mu_*$. The repelling slow manifold is treated similarly by considering the backwards flow. In this way, we determine the leading order of the splitting close to $\mu = \mu_*$ by the difference of the special solutions, see Proposition \ref{finalprop}.
In Section \ref{secdiff}, we complete the proof of Theorem \ref{TheoremSplitting} by setting up a normal form for the difference on $\kC_*$. Finally, in Sections \ref{proofprop} and \ref{secdelay}, we prove Proposition \ref{mainprop} and Corollary \ref{Delay}, respectively. 
\section{Basic properties}\label{basics}
In this section, we perform the local analysis of the singular point $\mu_*$ for \eqref{THIS} and \eqref{THIS2}.
First, we consider the invariant manifolds of the elliptic system. 
\begin{lemma}\label{Lemma21}
    Assume \ref{A1}. Then $\mu_*$ is a hyperbolic saddle of the elliptic system \eqref{THIS}, repeated here for convenience  
    \begin{equation*}
    \mu ' = -\i \overline{\lambda(\mu)}.
\end{equation*}
    Let
    $\mu = \mu_* + r \e^{\i\theta}$ and $\lambda_0=\lambda'(\mu_*) \neq 0$, recall \eqref{lambda0defn}.
    Then the tangents of the stable and unstable manifolds at $\mu_*$ make the following angles $\theta$ with the positive real axis 
    \begin{equation}\eqlab{thetasu}
    \theta^{S,E} = - \frac{1}{2}\arg(\lambda_0) - \frac{3\pi}{4}, \quad \theta^{U,E} = - \frac{1}{2}\arg(\lambda_0) - \frac{\pi}{4},
\end{equation}
respectively. Here $\arg(\lambda_0)$ denotes an argument of $\lambda_0\in \C\setminus\{0\}$.
\end{lemma}
\begin{proof}
    It suffices to consider the linearisation 
    \begin{equation}\eqlab{Esys}
    \dot \mu = -\i \overline{\lambda_0} \overline{\mu}.
\end{equation}
Writing $\mu = a + \i b$ one obtains a real linear system with the following eigenvalues  and eigenvectors of the linearisation: $\pm \vert\lambda_0\vert$ and $(a,b) = (- \real[\lambda_0],-\imag[\lambda_0]\mp \vert \lambda_0 \vert)$, respectively. This follows from a simple calculation and from these expression one easily obtains the explicit form for the angles.
\end{proof}

We illustrate the invariant manifolds and the angles $\theta$ associated with the tangents at $\mu_*$ in Figure \ref{Anglesa}. Notice that there is another branch of the stable and unstable manifold that correspond to $\theta^{S,E}+\pi$ and $\theta^{U,E}+\pi$, respectively. Moreover, we emphasize that the angle $\theta^{\kC}$, associated with the tangent of the contour $\kC_*$ at $\mu_*$, satisfies the inclusion:
\begin{align}\eqlab{thetaCin}
    \theta^{\kC}\in (\theta^{S,E},\theta^{U,E}) \quad \text{or} \quad \theta^{\kC}\in (\theta^{S,E}+\pi,\theta^{U,E}+\pi)
\end{align}
This follows from assumption \ref{B2} (recall also the discussion proceeding \ref{B2}), see Figure \ref{EBox}.

We will also use the following property.
\begin{lemma}\label{lemmaHam}
    The elliptic system has a conserved quantity given by
    \begin{equation}\eqlab{Hmu}
        \mathcal H(\mu) = \operatorname{Re} \left[\int_{\mu_*}^\mu \lambda(s) ds \right].
    \end{equation}
\end{lemma}
The result follows from a simple calculation, see also \cite[Lemma 4.1]{KRUPA20102841}.

 Next, we consider the corresponding hyperbolic system.
\begin{lemma}\label{Lemma22}
    Assume \ref{A1}. Then $\mu_*$ is a hyperbolic saddle of the hyperbolic system \eqref{THIS2}, repeated here for convenience 
    \begin{equation*}
    \dot \mu  = - \overline{\lambda(\mu)}.
\end{equation*}
    Let
   $        \mu = \mu_* + r \e^{\i\theta}.$
    Then the tangents of the stable and unstable manifolds at $\mu_*$ make the following angles $\theta$ with the positive real axis 
\begin{equation}\eqlab{HypAngle}
    \theta^{U,H} = -\frac{\arg(\lambda_0)}{2}-\frac{\pi}{2}, \quad \theta^{S,H} = - \frac{\arg(\lambda_0)}{2},
\end{equation}
respectively.
\end{lemma}
\begin{proof}
    The proof is identical to the proof of Lemma \ref{Lemma21}. Further details are therefore left out.
\end{proof}
We refer to Figure \ref{Anglesb} for an illustration.
Comparing the angles of Lemma \ref{Lemma21} and Lemma \ref{Lemma22}, we immediately see that
\begin{equation}\nonumber
     \frac{\theta^{S,E}+\theta^{U,E}}{2} = - \frac{1}{2}\arg ( \lambda_0) - \frac{\pi}{2} = \theta^{U,H}.
 \end{equation}
 The angles $\theta^{U,H}$  and $\theta^{U,H}+\pi$ will turn out to form natural barries for the continuation of attracting/repelling slow invariant manifolds near $\mu_*$, see Remark \ref{remthetaUH} below. In fact, the angles will be useful later on when we introduce certain polar coordinates associated with our blowup. 
 
\begin{figure}
    \centering
    \begin{subfigure}[b]{0.425\textwidth}
        \centering
        \includegraphics[width=\linewidth]{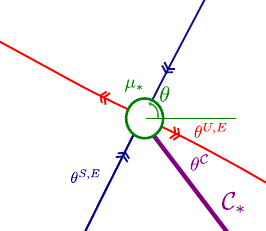}
        \caption{}
        \label{Anglesa}
    \end{subfigure}
    \hfill
    \begin{subfigure}[b]{0.485\textwidth}
        \centering
        \includegraphics[width=\linewidth]{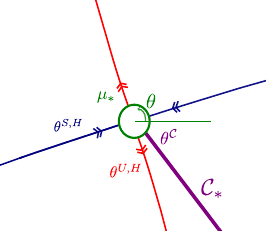}
        \caption{}
        \label{Anglesb}
    \end{subfigure}
    \caption{(a): The invariant manifolds of $\mu_*$ in the elliptic system. (b): The invariant manifolds of $\mu_*$ in the hyperbolic system. For simplicity, we only show the linear approximations, along with the angles $\theta$ from Lemma \ref{Lemma21} and \ref{Lemma22}. Notice that $\theta^{\kC}\in (\theta^{S,E},\theta^{U,E})$ or $\theta^{\kC}\in (\theta^{S,E}+\pi,\theta^{U,E}+\pi)$ (a consequence of assumption \ref{B2}, see main text for details). }
    \label{Anglesab}
\end{figure}


 

\section{Asymptotic series and blowup}\label{asympbu}
In this section, we consider formal asymptotic series for our slow manifolds solutions.
\subsection{Asymptotic series away from \texorpdfstring{$\mu_*$}{mu*}}
We first consider any compact subset $\Sigma_0 \subset \Sigma$ such that $\mu_*,\overline{\mu_*} \notin \Sigma_0$  and introduce the (column) vector of fast variables $$\zeta = (z,w).
$$
We then write \eqref{ComplexI} in the following vector form:
\begin{equation}\eqlab{VecSys}
    \begin{cases}
     \dot{\zeta} = \operatorname{diag}(\lambda(\mu),\nu(\mu))\zeta + \epsilon H(\zeta,\mu,\epsilon),\\
     \dot{\mu} = \epsilon,
    \end{cases}
\end{equation}
with $H = (F,G)$. 
We notice that $\operatorname{diag}(\lambda(\mu),\nu(\mu))$ is invertible for all $\mu \in \Sigma_0$, cf. assumption \ref{A1} and Lemma \ref{lemmasym}.

Recall from the introduction, that we fix slow invariant manifolds  as graphs $\zeta=\xi^{a,r}(\mu,\epsilon)$ over an interval containing $\mu_s$ and $\mu_u$, respectively. The following is well known:
\begin{lemma}\label{FormS1}
    There exists a unique formal invariant manifold solution of \eqref{VecSys}:
    \begin{equation}\eqlab{FS1}
    \zeta = \sum_{n=1}^{\infty}\xi_n(\mu)\epsilon^n
\end{equation}
with $\xi_n:\Sigma_0\to \mathbb C^2$ analytic for all $n \in \N$. The functions $\xi^a$ and $\xi^r$ are each asymptotic to \eqref{FS1} on their respective domains.
\end{lemma}

\begin{proof}
    Since the diagonal matrix in \eqref{VecSys} is invertible for all $\mu \in \Sigma_0$, the result follows from Proposition 2.1 of \cite{de2020a}.
\end{proof}
\subsection{Blowup}
The asymptotic series \eqref{FS1} is not well-defined near $\mu_*$. To study the system near the degenerate point $\mu_*$, we consider the following blowup transformation

\begin{equation}\eqlab{BlowUp}
    \begin{cases}
         z  =r \breve{z},\\
        w = r \breve{w},\\
        \mu = \mu_*+r\breve{\mu},\\
        \epsilon =r^2\breve{\epsilon},
    \end{cases}
\end{equation}
with $r \geq 0, (\breve{z},\breve{w},\breve{\mu},\breve{\epsilon}) \in \S^3\subset \mathbb C^4$,  and where $\S^3$ denotes the complex 3-sphere. We will consider the following charts
\begin{align}
    \text{$\vert\breve{\mu}\vert=1$-chart}:&
    \begin{cases}
         z  =r_1 z_1,\\
        w = r_1 w_1,\\
        \mu = \mu_*+r_1\e^{i\theta},\\
        \epsilon =r_1^2\epsilon_1,
    \end{cases}\eqlab{chart2}
    \\
    \text{$\breve \mu=1$-chart:}&
    \begin{cases}
         z  =r_2 z_2,\\
        w = r_2 w_2,\\
        \mu = \mu_*+r_2,\\
        \epsilon =r_2^2\epsilon_2,
    \end{cases}\eqlab{chart1}
\end{align}
that we (as indicated) refer to as the $\vert\breve{\mu}\vert=1$-chart and $\breve \mu=1$-chart, respectively.

In \eqref{chart1}, the chart-specific coordinates $(z_2,w_2,r_2,\epsilon_2)$ are complex variables. In contrast, we consider the chart-specific coordinates $(\theta,r_1,\epsilon_1)$ of the $\vert\breve{\mu}\vert=1$-chart, see \eqref{chart2}, to be real. Notice that $(r_1,\theta)$ are polar coordinates based at $\mu_*$ in the complex $\mu$-plane.


The change of coordinates between the two charts is given by
\begin{align}\eqlab{chanco}
       \begin{cases}
        z_1 = \e^{i\theta}z_2,\\
        w_1 = \e^{i\theta}w_2,\\
        r_1 = r_2\e^{-i\theta},\\
        \epsilon_1 = \epsilon_2\e^{2i\theta}.
       \end{cases}
    \end{align}
    This is obtained by equating  \eqref{chart1} with \eqref{chart2}. 

In the following, we first look for asymptotic series in the $\breve \mu = 1$-chart. The coordinates $(z_2,w_2,r_2,\epsilon_2)$ are ideally suited for this purpose (\cite[Section 4]{kristiansenzerophopf2} uses similar coordinates for the same objective). Subsequently, we translate the asymptotic series to the $ \vert \breve \mu \vert = 1$-chart by the change of coordinates \eqref{chanco}. These series will be fundamental in Section 4, when we use the $(z_1,w_1,\theta,r_1,\epsilon)$-coordinates of the $\vert\breve{\mu}\vert=1$-chart to extend the slow invariant manifolds defined by $\xi^{a,r}$ up close to $\mu_*$. Here we will also use a desingularization corresponding to division of the right hand side by (the real quantity) $r_1$.
\subsection{Asymptotic series in \texorpdfstring{$\breve \mu = 1$}{\breve \mu = 1}-chart}
Consider the coordinates of the $\breve \mu = 1$-chart \eqref{chart1}. Performing the change of coordinates in \eqref{ComplexI}, we find the system
\begin{equation}\eqlab{BlowupChart1}
    \begin{aligned}
        & z_2' = \lambda(\mu_*+r_2) z_2 + r_2\epsilon_2 F(r_2z_2,r_2w_2,\mu_*+r_2,r_2^2\epsilon_2) - \epsilon_2 r_2 z_2, \\
        & w_2' = \nu(\mu_*+r_2) w_2 + r_2\epsilon_2 G(r_2z_2,r_2w_2,\mu_*+r_2,r_2^2\epsilon_2) - \epsilon_2 r_2 w_2, \\
        & r_2' = r_2^2 \epsilon_2,\\
        & \epsilon_2' = -2 r_2 \epsilon_2^2.
    \end{aligned}
\end{equation}
We look for invariant manifold solutions of \eqref{BlowupChart1} as graphs over $(r_2,\epsilon_2)$
\begin{equation}\eqlab{SlowManifoldRE1}
    (z_2,w_2) = (z_2,w_2)(r_2,\epsilon_2).
\end{equation}
Inserting \eqref{SlowManifoldRE1} into \eqref{BlowupChart1}, we obtain the following invariance equations 
\begin{equation}\eqlab{PDEs}
    \resizebox{0.95\linewidth}{!}{$
    \begin{aligned}
    &(\partial_{r_2} z_2)r_2\epsilon_2-2(\partial_{\epsilon_2} z_2)\epsilon_2^2=(\hat{\lambda}(r_2)-\epsilon_2)z_2+\epsilon_2F(r_2z_2,r_2w_2,\mu_*+r_2,r_2^2\epsilon_2),\\
    &(\partial_{r_2} w_2)r_2^2\epsilon_2-2(\partial_{\epsilon_2} w_2)r_2\epsilon_2^2=({\nu}(r_2+\mu_*)-r_2\epsilon_2)w_2+r_2\epsilon_2G(r_2z_2,r_2w_2,\mu_*+r_2,r_2^2\epsilon_2).
\end{aligned}
$}
\end{equation}

The first equation has been desingularized through division by $r_2$. This is made possible by the fact that $\mu_*$ is a root of $\lambda$. In particular, we have introduced the function: 
\begin{equation}\eqlab{hatlambda}
    \hat{\lambda}(r_2) = \frac{\lambda(\mu_*+r_2)}{r_2},
\end{equation}
having an analytic extension to $r_2=0$ with $$\hat{\lambda}(0)=\lambda_0\neq0,$$ recall assumption \ref{A1}.  Note that, this desingularization is not possible in the equation for $w_2$ as
\begin{equation}\eqlab{nu0}
    \nu(\mu_*) \neq 0,
\end{equation}
cf. \eqref{symFL} and assumption \ref{A1}.
We will look for two types of formal invariant manifold solutions of the PDE-system \eqref{PDEs}:
\begin{align}
    & \text{type I:} \quad (z_2,w_2) = \sum_{n=1}^\infty \xi_{n2}(r_2)\epsilon_2^n,\eqlab{typeIser}\\
    & \text{type II:} \quad (z_2,w_2) = \sum_{n=0}^\infty m_{n2}(\epsilon_2)r_2^n.\eqlab{typeIIser}
\end{align}
In other words, type I series are formal series in $\epsilon_2$ with $r_2$-dependent coefficients, whereas for type II it is the other way around, i.e. formal series in $r_2$ with $\epsilon_2$-dependent coefficients.

Let $B_\delta\subset \mathbb C$ denote the open disc centered at the origin of radius $\delta>0$.
\begin{lemma}\label{T1}
    There exists a unique formal series of type I. In particular, there is a $\delta>0$ small enough such that $ \xi_{n2}:B_\delta \mapsto \C^2$ are analytic functions for all $n \in \N$. 
\end{lemma}
\begin{proof}
    We write $\zeta_2=(z_2,w_2)$. Then the PDE-system takes the following vector form:
    \begin{equation}\eqlab{inveqn2}
    \begin{aligned}
        \epsilon_2 
        \begin{pmatrix}
            r_2&0\\
            0&r_2^2
        \end{pmatrix}
        (\partial_{r_2} \zeta_2)&= A(r_2)\zeta_2\\
        &+\epsilon_2 \begin{pmatrix}
            1&0\\
            0&r_2
        \end{pmatrix} \left( H(r_2\zeta_2,\mu_*+r_2,r_2^2\epsilon_2) +
2\epsilon_2(\partial_{\epsilon_2}\zeta_2)-\zeta_2\right),
    \end{aligned}
    \end{equation}
    with
    \begin{align*}
        & A(r_2) := 
        \operatorname{diag}(\hat{\lambda}(r_2),\nu(\mu_*+r_2)).
    \end{align*}
    The point is now that $A(r_2)$ is regular. We can therefore apply Proposition 2.1 in \cite{de2020a} to the system \eqref{PDEs} (the proof proceeds completely analogously). This leads to the existence and uniqueness of formal series of type I.
\end{proof}

Next for type II series, we let
    \begin{equation}\eqlab{Sector}
        S(\psi,\chi,\delta) =\left\{x \in \C\,:\, 0<\vert x\vert<\delta, \vert \arg(x)-\psi \vert < \frac{\pi+\chi}{2} \right\},
    \end{equation}
    denote the sector of opening $\pi + \chi$, centered along the direction $\psi$, and with radius $\delta$. Moreover, we define 
    \begin{equation}
        \psi^S: = - 2\theta^{S,E} = \arg(\lambda_0) - \frac{\pi}{2}, \quad \psi^U: = - 2\theta^{U,E} = \arg(\lambda_0) + \frac{\pi}{2},\eqlab{psiSU}
    \end{equation}
    with equality understood $\operatorname{mod}(2\pi)$.  For details on summability, we refer to \cite{bonckaert2008a,de2020a}.
\begin{lemma}\label{T2}
    Suppose that $\psi \in \{\psi^{S},\psi^{U}\}$ and fix any $\chi \in (0,\pi)$.
    Then for $\delta>0$ small enough, there exists a formal series of \eqref{inveqn2} of type II: $m_{n2}=m^{\psi}_{n2}$, $n\in \mathbb N_0$, with
    \begin{equation*}
        m^{\psi}_{n2}:S(\psi,\chi,\delta) \mapsto \C^2,
    \end{equation*}
   being analytic functions for all $n \in \N_0$. In further details, each $m^{\psi}_{n2}$, $n\in \mathbb N_0$, is the 1-sum of a Gevrey-1 series in the direction $\psi$.
\end{lemma}
\begin{proof}
We consider \eqref{inveqn2} with subscripts dropped and look for formal series solutions of the form
    \begin{align}\nonumber
        &\zeta =\sum_{n=0}^\infty m_n(\epsilon)r^n,\quad m_n = (z_n,w_n). 
    \end{align}
          Then $W(rz,rw,\mu_*+r,r^2\epsilon)$, $W\in \{F,G\}$, are (by composition) formal series in $r$ with $\epsilon$-dependent coefficients. 
          By the Faa di Bruno formula (\cite[Theorem 2.1]{constantine1996a}), we have:
    \begin{align*}
        W(rz,rw,\mu_*+r,r^2\epsilon) &= W(0,0,\mu_*,0)+\sum_{n=1}^\infty W_n(\epsilon,m_0,...,m_{n-1})r^n,
        \end{align*}
    where $W_n$ is degree $(n-1)$ polynomial with respect to $(m_1,...m_{n-1})$ for $\epsilon,m_0$ fixed. We also write 
     \begin{align}
        \begin{aligned}
           \hat \lambda(r) &= \sum_{n=0}^\infty \lambda_{n}r^{n},\quad 
        \nu(\mu_*+r) = \sum_{n=0}^\infty \nu_n r^n,
        \end{aligned}
    \end{align}
    which are convergent series. Notice that $\nu_0\ne 0$ (cf. \eqref{nu0}).
    By inserting the type II formal series into the PDE-system \eqref{PDEs} and collecting terms in $\epsilon$, we find
    \begin{equation}
       \nonumber 2\epsilon^2z_0'(\epsilon) =-(\lambda_0-\epsilon)z_0(\epsilon)-\epsilon F_0, \quad w_0(\epsilon) \equiv 0,
    \end{equation}
    and
    \begin{align}
        2 \epsilon^2 z'_n(\epsilon) &=-(\lambda_0-(n+1)\epsilon)z_n(\epsilon)-\epsilon F_n(\epsilon,m_0,...,m_{n-1})-\sum_{i=1}^{n}\lambda_{i}z_{n-i}(\epsilon), \nonumber \\
        \nu_0 w_n(\epsilon) & =-2w'_{n-1}\epsilon^2+n\epsilon  w_{n-1}(\epsilon)-\epsilon G_{n-1}(\epsilon,m_0,...,m_{n-2})-\sum_{i=1}^{n}\nu_i w_{n-i}(\epsilon), \nonumber
    \end{align}
    for any $n \in \N$ with $W_0(\epsilon,m_0,...,m_{n-1}) = W_0$, $W\in \{F,G\}$. Since $\nu_0 \neq 0$, we can solve the linear equation for $w_n(\epsilon)$ recursively (starting from $w_0(\epsilon)\equiv 0$). In this way, $w_n(\epsilon)$ becomes a function of $(\epsilon,z_0,...,z_{n-2})$ for all $n \in \N$, resulting in the following equations for $z_n(\epsilon)$:
    \begin{equation}\eqlab{z0}
        2\epsilon^2z_0'(\epsilon) = - (\lambda_0-\epsilon)z_0(\epsilon)-\epsilon F_0,
    \end{equation}
    and 
    \begin{equation}\eqlab{nequation}
        2\epsilon^2z_n'(\epsilon) = - (\lambda_0-(n+1)\epsilon)z_n(\epsilon)-\epsilon F_n(\epsilon,z_0,...,z_{n-1})-\sum_{i=1}^n\lambda_{i}z_{n-i}(\epsilon),
    \end{equation}
    with $F_n$, $n\in \mathbb N$, being an analytic function of $\epsilon,z_1,...,z_{n-1}$ for each $n\in \mathbb N$. These equations for $z_n$ can be solved by induction on $n$ using the Borel-Laplace approach in \cite[Section 6]{bonckaert2008a}, see also \cite[Appendix A]{kriszm1}. The base  case $n=0$ of the induction step proceeds as follows: We write
    \begin{equation}\nonumber
        z_0 = \mathcal{L}[Z_0], 
    \end{equation}
    where $\mathcal{L}$ is the Laplace transform:
    \begin{equation}\nonumber
        \mathcal{L}[Z_0](\epsilon) = \int_{0}^{\infty \e^{\i\psi}} Z_0(s) \e^{-s\epsilon^{-1}} ds.
    \end{equation}
    By the application of the inverse Laplace (Borel transform, see \cite{bonckaert2008a}), this brings \eqref{z0} into the form
    \begin{equation}\eqlab{LZ}
        2sZ_0(s) = - \lambda_0 Z_0(s) + (1 \star Z_0)(s) - F_0,
    \end{equation}
    where $\star$ denotes the convolution. Here we have used that
    \begin{equation}\nonumber
        \epsilon^2 \frac{d}{d \epsilon}\mathcal{L}[Z_0](\epsilon) = \mathcal{L}[s.Z_0](\epsilon),
    \end{equation}
    where $s.Z_0$ denotes the function $s \mapsto sZ_0(s)$. The solution of \eqref{LZ} is given by
    \begin{equation}\eqlab{Basestep1}
        Z_0(s) = 
        -(2s\lambda_0 + \lambda_0^2)^{-\frac{1}{2}}F_0
    \end{equation}
    To obtain this, we have differentiated \eqref{LZ} with respect to $s$ using $(1\star Z_0)'(s) = Z_0(s)$, and solved the resulting differential equation with the initial condition $Z_0(0)=-F_0/\lambda_0$.
        In further details, we fix a branch cut along the ray emanating from the point $s = - \frac{\lambda_0}{2}$ and direction given by $\arg(s) = \arg(-\lambda_0) =\arg(\lambda_0) - \pi$, with $\arg(\lambda_0)$ being a fixed argument of $\lambda_0\in \C\setminus\{0\}$. Hence, we define the square root in \eqref{Basestep1} as
        \begin{align}\eqlab{sqrt}
            (2s\lambda_0 + \lambda_0^2)^{\frac{1}{2}} := \sqrt{\vert 2s\lambda_0 + \lambda_0^2\vert} \exp\left({\frac{\i}{2}\left(\arg(u)+\arg(\lambda_0)\right)}\right),
        \end{align}
        with $$u=2s+\lambda_0,$$ and $$\arg(s)\in (\arg(\lambda_0)-\pi,\arg(\lambda_0)+\pi],$$ such that $(\lambda_0^2)^{\frac{1}{2}}=\lambda_0$ and $$Z_0(0)=-F_0/\lambda_0.$$
        In this way, we obtain our solution $z_0(\epsilon)$ by Laplace transforming in either direction $\psi \in \{ \psi^S,\psi^U\}$:
    \begin{equation}\eqlab{z0lo}
        z_0(\epsilon) = -F_0\int_0^{\infty \e^{\i\psi}} (2s\lambda_0+\lambda_0^2)^{-\frac{1}{2}}\e^{-s \epsilon^{-1}}ds.
    \end{equation}
    Notice that the angles $\psi^{S}$ and $\psi^{U}$ are bounded away from the singular direction $\arg(s) = \arg(\lambda_0)-\pi$ by \eqref{psiSU} (they are rotated by $90^\circ$ counter-clockwise and clockwise, respectively), see Figure \ref{Laplacefig}.
    By analytic continuation,  \eqref{z0lo} defines an analytic function of $\epsilon$ on $S(\psi,\chi,\delta)$ with $\chi \in (0, \pi)$, see \cite[Proposition 3]{bonckaert2008a}. For the induction step, we suppose that $$z_i = \mathcal L[Z_i]:S(\varphi,\chi,\delta)\to \mathbb C, \quad i \in \{0,\ldots,n-1\},$$ and Borel-transform \eqref{nequation}. This gives
    \begin{equation}\nonumber
        2sZ_n(s) =  - \lambda_0 Z_n(s) + (n+1)(1 \star Z_n)(s)-F_n^{\star}(s)-\sum_{i=1}^n\lambda_i Z_{n-i}(s),
    \end{equation}
    where $F_n^\star$ can be expressed as the convolution product of the known functions $Z_0$,$\ldots$, $Z_{n-1}$ using the analyticity of $F_n$, see \cite[Proposition 5]{bonckaert2008a}. This equation can solved in a similar way to \eqref{LZ} and upon Laplace transforming the resulting $Z_n$, we obtain the desired analytic function $z_n:S(\psi,\chi,\delta)\to \mathbb C$.
\end{proof}
\begin{figure}
    \centering
    \begin{subfigure}[b]{0.49\textwidth}
        \centering
        \includegraphics[width=\linewidth]{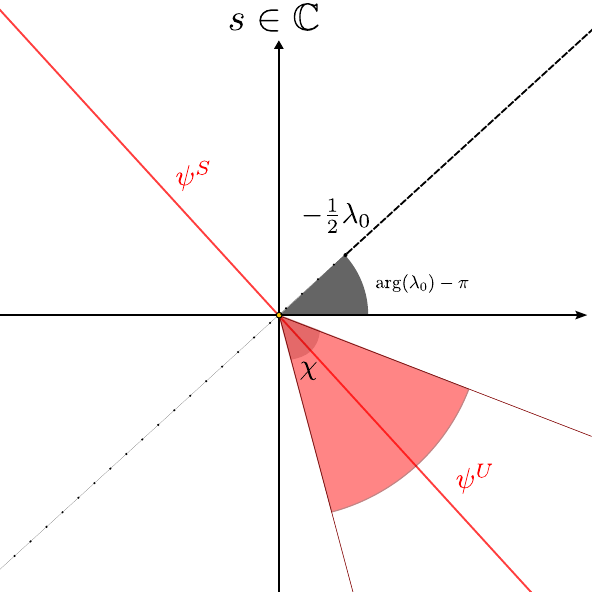}
       \caption{} 
        \label{Sektor1}
    \end{subfigure}
    \hfill
    \begin{subfigure}[b]{0.49\textwidth}
        \centering
        \includegraphics[width=\linewidth]{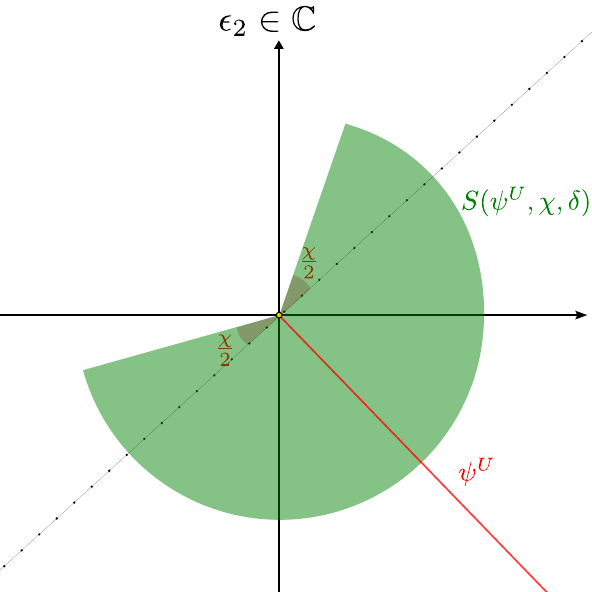}
        \caption{} 
        \nonumber
    \end{subfigure}
    \caption{Illustration of the domains used in the solution of \eqref{LZ} in the proof of Lemma \ref{T2}.} 
    \label{Laplacefig}
\end{figure}
\begin{remark}\label{remthetaUH}
    The singular direction $\psi$, associated with the Laplace transform, is illustrated in Figure \ref{Sektor1} as the ray making an angle $\arg(\lambda_0)-\pi$ with the positive real axis. This angle corresponds to:
    \begin{equation}
        \psi = - 2\theta^{U,H},
    \end{equation} 
        recall Lemma \ref{Lemma22}.
\end{remark}
Consequently, we have proven the existence of type I and II formal series. The two series coincide as formal series in $(r_2,\epsilon_2)$:

\begin{lemma}\label{Doublesseriescof}
Consider the formal series in Lemma \ref{T1} and \ref{T2} and expand the coefficients $\xi_{n2}$ and $m_{n2}$  for any $n\in \mathbb N_0$ as formal series of $r_2$ and $\epsilon_2$, respectively. The resulting two formal series in $(r_2,\epsilon_2)$ are identical.
\end{lemma}
\begin{proof}
    Follows from uniqueness of the formal series. 
\end{proof}

As a final result of this section, we give the explicit leading order of the difference of the asymptotic series in Lemma \ref{T2}. 
\begin{figure}
    \centering
    \includegraphics[width=0.55\linewidth]{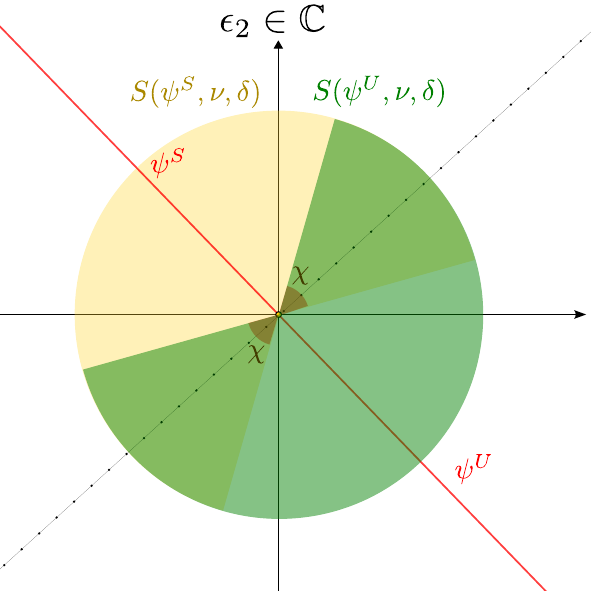}
    \caption{Illustration of the sectors $S(\psi^S,\chi,\delta)$ and $S(\psi^U,\chi,\delta)$ in the complex $\epsilon_2$-plane.}
    \label{Sektor1ab}
\end{figure}
\begin{figure}
    \centering
    \includegraphics[width=0.55\linewidth]{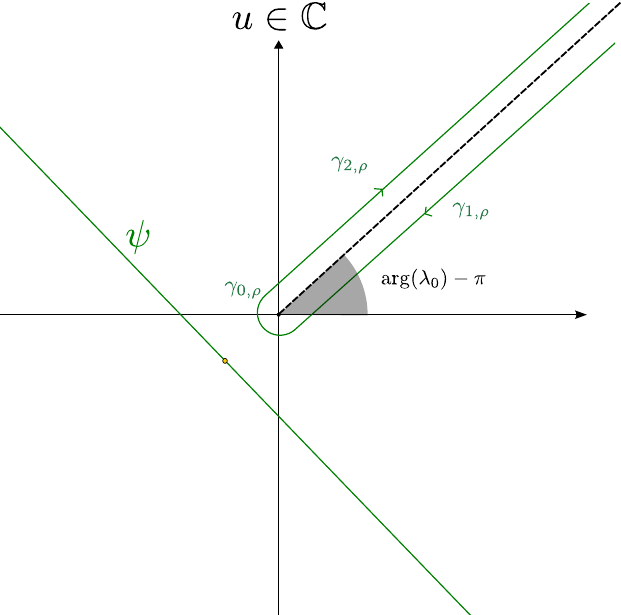}
    \caption{The integration contour used in the proof of Lemma \ref{SN1}.}
    \label{IntegrationContour}
\end{figure}

\begin{lemma}\label{SN1}
    The following holds
    \begin{equation}\nonumber
    \begin{aligned} 
        &m_{02}^{\psi^S}(\epsilon_2) - m_{02}^{\psi^U}(\epsilon_2) = \left(-\frac{\i F_0}{\sqrt{\vert \lambda_0 \vert}}\e^{-\i \arg(\lambda_0)} \e^{\frac{1}{2\epsilon_2}\lambda_0} \sqrt{2\pi\epsilon_2 \e^{-\i( \arg(\lambda_0)+\pi)}},0\right),
    \end{aligned} 
    \end{equation}
    for any  $\epsilon_2$ in the local  sector centered along the direction $\arg(\lambda_0)-\pi$, with opening $\chi$ and radius  $\delta>0$:
    $$\epsilon_2 \in S(\arg(\lambda_0)-\pi,\chi-\pi,\delta)\subset S(\psi^S,\chi,\delta) \cap S(\psi^U,\chi,\delta),$$
    recall \eqref{Sector}.
\end{lemma}
\begin{proof}
    From  Lemma \ref{T2}, we have
    \begin{equation}\nonumber
        z_{02}^\psi(\epsilon_2) = -F_0\int_0^{\infty \e^{\i\psi}} (2s\lambda_0+\lambda_0^2)^{-\frac{1}{2}}\e^{-s \epsilon_2^{-1}}ds, \quad w_{02}^\psi(\epsilon_2) = 0,
    \end{equation}
    leading to the difference:
    \begin{equation}
    \begin{aligned}\eqlab{z02w02}
        z_{02}^{\psi^S}(\epsilon_2) - z_{02}^{\psi^U}(\epsilon_2) &= -F_0\int_{\infty\e^{\i\psi^{U}}}^{\infty\e^{\i\psi^{S}}} (2s\lambda_0+\lambda_0^2)^{-\frac{1}{2}}\e^{-s \epsilon_2^{-1}}ds,\\
        w_{02}^{\psi^S}(\epsilon_2)-w_{02}^{\psi^U}(\epsilon_2) &= 0,
    \end{aligned}
    \end{equation}
    for any $\epsilon_2 \in S(\arg(\lambda_0)-\pi,\chi-\pi,\delta)$.
   Now, define $I_\gamma$ as
   \begin{equation}\nonumber
              I_{\gamma} = \int_{\gamma} (2s\lambda_0+\lambda_0^2)^{-\frac{1}{2}}\e^{-s \epsilon_2^{-1}}ds,
   \end{equation}
   with $\gamma$ a contour in the complex $s$-plane connecting the two points at infinity (${\infty\e^{\i\psi^{U}}}$ and ${\infty\e^{\i\psi^{S}}}$) and avoiding the branch cut along the ray defined by $\arg(\lambda_0)-\pi$, see Figure \ref{Laplacefig}.
       By performing the change of coordinates $2s + \lambda_0 = u$, we obtain:
    \begin{equation}\nonumber
        I_\gamma = \frac{1}{2}\int_\gamma (\lambda_0u)^{-\frac{1}{2}} \e^{\frac{\lambda_0-u}{2\epsilon_2}}du,
    \end{equation}
    where we continue to denote the contour by $\gamma$ in the new coordinates. 
    We now deform $\gamma$ to the path shown in the Figure \ref{IntegrationContour}, avoiding the branch cut. The integration is subsequently divided into 3 domains $\gamma =  \gamma_{0,\rho} \cup \gamma_{1,\rho} \cup \gamma_{2,\rho}$ where $\gamma_{0,\rho}$ is the half-circle with radius $\rho$ around the branch point, and $\gamma_{1,\rho},\gamma_{2,\rho}$ are the corresponding contours that are parallel but $\pm\rho$ away from the branch cut, corresponding to $\arg(u) = \arg(\lambda_0)+\pi$ and $\arg(u) = \arg(\lambda_0)-\pi$, respectively.
     
    Firstly, we let $I_{\gamma_{0}}$ denote the contribution to $I_\gamma$ from $\gamma_{0,\rho}$ given by $u = \rho \e^{i\theta}$, $\theta\in J$, with $J\subset \mathbb R$ an appropriate interval of length $\pi$. We directly obtain 
    \begin{equation}\nonumber
        I_{\gamma_{0}} = \int_{J} \frac{1}{2} (\lambda_0\rho \e^{\i \theta})^{-\frac{1}{2}} \e^{\frac{1}{2}(\lambda_0-\rho \e^{\i \theta}) \epsilon_2^{-1}} \rho \i \e^{\i \theta} d\theta = \mathcal{O(\sqrt{\rho})},
    \end{equation}
        leading to a vanishing contribution in the $\rho \rightarrow 0$ limit.
  
    Next, we let $I_{\gamma_1}$ denote the contribution to $I_\gamma$ from $\gamma_{1,\rho}$ given by $$u=r\e^{i(\arg(\lambda_0)+\pi)}, \quad r\ge 0,$$ in the $\rho \rightarrow 0$ limit. 
    This leads to 
    \begin{equation}\nonumber
        I_{\gamma_1} =  \frac{\i}{2\sqrt{\vert \lambda_0 \vert}}\e^{-\i \arg(\lambda_0)} \e^{\frac{1}{2\epsilon_2}\lambda_0}\int_{0}^\infty r^{-\frac{1}{2}}\exp\left(-\frac{1}{2\epsilon_2}\e^{\i(\arg(\lambda_0)+\pi)}r\right)dr.
    \end{equation}
    Here we have used the definition of the root in \eqref{sqrt}.
    $I_{\gamma_1}$ is related to $I_{\gamma_2}$ (being the contribution to $I_\gamma$ from $\gamma_{2,\rho}$) by replacing the argument $\arg(\lambda_0)+\pi$ by $\arg(\lambda_0)-\pi$  and reversing the direction of integration. We find
    \begin{equation}\eqlab{Igamma}
    \begin{aligned}
        I_\gamma = I_{\gamma_1}+I_{\gamma_2} &= \frac{\i}{\sqrt{\vert \lambda_0 \vert}}\e^{-\i \arg(\lambda_0)} \e^{\frac{1}{2\epsilon_2}\lambda_0}\int_{0}^\infty r^{-\frac{1}{2}}\exp\left(-\frac{1}{2\epsilon_2}\e^{\i(\arg(\lambda_0)+\pi)}r\right)dr\\
        & = \frac{\i}{\sqrt{\vert \lambda_0 \vert}}\e^{-\i \arg(\lambda_0)} \e^{\frac{1}{2\epsilon_2}\lambda_0} \sqrt{2\pi\epsilon_2 \e^{-\i( \arg(\lambda_0)+\pi)}},
    \end{aligned}\end{equation}
    as $\rho\to 0$. Here the square root on the right hand side is the usual square root with the branch cut along the negative real axis. Notice in particular that the integral in \eqref{Igamma}  is convergent since 
    \begin{align*}
        \cos\left(\arg\left(\frac{1}{2\epsilon_2} \e ^{\i\left(\arg(\lambda_0)+\pi\right)}\right)\right) > \cos\left(\frac{\chi}{2}\right) >0,
    \end{align*}
    for any
    $$\epsilon_2\in S(\arg(\lambda_0)-\pi,\chi-\pi,\delta).$$
    Here we have used that $\chi\in (0,\pi)$.
In fact, the argument of the square root: $$2\pi \epsilon_2 \e^{-\i( \arg(\lambda_0)+\pi)},$$ in \eqref{Igamma} is real and positive for $\arg(\epsilon_2)=\arg(\lambda_0)-\pi$.   Finally using \eqref{z02w02}, we obtain
    \begin{equation}\nonumber
        z_{02}^{\psi^S}(\epsilon_2)-z_{02}^{\psi^U}(\epsilon_2) = -\frac{\i F_0}{\sqrt{\vert \lambda_0 \vert}}\e^{-\i \arg(\lambda_0)} \e^{\frac{1}{2\epsilon_2}\lambda_0} \sqrt{2\pi\epsilon_2 \e^{-\i( \arg(\lambda_0)+\pi)}}, 
    \end{equation}
    which completes the proof.
\end{proof}

\subsection{Asymptotic series in the \texorpdfstring{$\vert \breve \mu \vert = 1$}{!} chart}\label{asymp1}
In the  $\vert \breve \mu \vert=1$-chart, we obtain the following  system from \eqref{ELI}:
\begin{equation}\eqlab{bchart2}
\begin{aligned}
    &z_1'=r_1[-\i\vert\hat \lambda\vert^2z_1-\i\e^{-\i\theta}\overline{\hat \lambda}\epsilon_1F-\epsilon_1z_1\real[\i\hat \lambda\e^{\i\theta}]],\\
    &w_1' = -\i\e^{-\i\theta}\overline{\hat \lambda}\nu w_1-\i\e^{-\i\theta}\overline{\hat \lambda}r_1\epsilon_1G-r_1\epsilon_1w_1\real[\i\hat \lambda\e^{\i\theta}],\\
    &\theta' =\epsilon_1r_1\real[-\hat \lambda\e^{\i\theta}],\\
    &r_1 '= r_1^2\epsilon_1\real[\i\hat \lambda\e^{\i\theta}],\\
    &\epsilon_1'= -2\epsilon_1^2r_1\real[\i\hat \lambda\e^{\i\theta}],
\end{aligned}
\end{equation}
using \eqref{chart1} as a change of coordinates,
where we (for simplicity) write: 
\begin{equation}\eqlab{simpl}
    \begin{aligned}
        & \hat \lambda = \hat \lambda(r_1\e^{\i\theta}),\\
        & \nu = \nu(\mu_*+r_1\e^{\i\theta}),\\
        & F = F(r_1z_1,r_1w_1,r_1^2\epsilon_1,\mu_*+r_1\e^{\i\theta}),\\
        &G = G(r_1z_1,r_1w_1,r_1^2\epsilon_1,\mu_*+r_1\e^{\i\theta}).
    \end{aligned}
\end{equation}
In the derivation of \eqref{bchart2}, we have divided the right hand side by the common factor $r_1$ (desingularization). 
Analogously to the previous section, we now look for formal invariant manifolds  $(z_1,w_1)(\theta,r_1,\epsilon_1)$ of \eqref{bchart2} that are formal series with respect to $r_1$ and $\epsilon_1$ with analytic $\theta$-dependent coefficients. They satisfy the invariance equation in a formal sense:
\begin{equation}\eqlab{mu1i}
    \begin{aligned}
        &(\partial_\theta z_{1})\epsilon_1\real[-\hat \lambda\e^{\i\theta}]+(\partial_{r_1} z_{1})\epsilon_1r_1\real[\i\hat \lambda\e^{\i\theta}]-2(\partial_{\epsilon_1}z_{1})\epsilon_1^2\real[\i\hat \lambda\e^{\i\theta}] \\
        & = -\i\vert\hat \lambda\vert^2z_1-\i\e^{-\i\theta}\overline{\hat \lambda}\epsilon_1F-\epsilon_1z_1\real[\i\hat \lambda\e^{\i\theta}].\\
        &(\partial_\theta w_{1})\epsilon_1r_1\real[-\hat \lambda\e^{\i\theta}]+(\partial_{r_1} w_{1})\epsilon_1r_1^2\real[\i\hat \lambda\e^{\i\theta}]-2(\partial_{\epsilon_1})w_{1}\epsilon_1^2r_1\real[\i\hat \lambda\e^{\i\theta}] \\
        & = -\i\e^{-\i\theta}\overline{\hat \lambda}\nu w_1-\i\e^{-\i\theta}\overline{\hat \lambda}r_1\epsilon_1G-r_1\epsilon_1w_1\real[\i\hat \lambda\e^{\i\theta}].\\
    \end{aligned}
\end{equation}
For later reference, we write $\zeta_1 = (z_1,w_1)$ and introduce the related PDE-system:
\begin{equation}\eqlab{Zeta1inv}
    \begin{aligned}
    &(\partial_\theta \zeta_1) \epsilon_1r_1\real[-\hat \lambda\e^{\i\theta}]+(\partial_{r_1} \zeta_1)\epsilon_1r_1^2\real[\i\hat \lambda\e^{\i\theta}]-2(\partial_{\epsilon_1} \zeta_{1})\epsilon_1^2r_1\real[\i\hat \lambda\e^{\i\theta}] \\
        & = P(\theta,r_1) \zeta_1-\i\e^{-\i\theta}\overline{\hat \lambda}r_1\epsilon_1H-r_1\epsilon_1\zeta_1\real[\i\hat \lambda\e^{\i\theta}].
\end{aligned}
\end{equation}
Here we have defined
\begin{equation}\eqlab{P}
P(\theta,r_1)= \operatorname{diag}(-\i r_1\vert\hat \lambda\vert^2,-\i\e^{-\i\theta}\overline{\hat \lambda}\nu),
\end{equation}
and $H = (F,G)$, recall \eqref{simpl}. We have the following preliminary result. 

\begin{lemma}\label{FormalSeries1}
Let $(z_2,w_2)(r_2,\epsilon_2)$ be the formal series of either type I of Lemma \ref{T1} or type II of Lemma \ref{T2}. Then
    \begin{equation}\nonumber
        (z_1,w_1)(\theta,r_1,\epsilon_1) = \e^{\i\theta}(z_2,w_2)(r_1\e^{\i\theta},\epsilon_1\e^{-2\i\theta}),
    \end{equation}
    defines a formal series solution of \eqref{mu1i}.
\end{lemma}
\begin{proof}
    The system \eqref{PDEs} is equivalent with \eqref{mu1i} under the change of coordinates \eqref{chanco}. The existence of the formal series in the $\vert \breve \mu \vert = 1$-chart then follows from applying the change of coordinates to the unique formal series of Lemma \ref{T1} or Lemma \ref{T2}.
\end{proof}
If we take the type I series, then we obtain formal series in $\epsilon_1$ with analytic $(\theta,r_1)$-dependent coefficients:
\begin{equation}\eqlab{FRØ1}
    (z_1,w_1)(\theta,r_1,\epsilon_1) = \sum_{n=1}^{\infty} \e^{\i(-2n+1)\theta}\xi_{n2}(r_1 \e^{\i\theta}) \epsilon_1^n.
\end{equation}
\begin{lemma}\label{FSU3}
    The two formal series \eqref{FS1} and \eqref{FRØ1} agree under the change of coordinates \eqref{chart2} with:
    \begin{equation}
\nonumber\xi_n(\mu_*+r_1\e^{\i\theta})r_1^{2n-1} = \e^{i(-2n+1)\theta} \xi_{n2}(r_1\e^{\i\theta}).
    \end{equation}
\end{lemma}
\begin{proof}
    We write the formal series \eqref{FRØ1} in the coordinates of $(z,w,\mu,\epsilon)$ using \eqref{chart2}. Since the coefficients of \eqref{FS1} are unique, the result follows.
\end{proof}
Analogously for type II series, then we obtain formal series in $r_1$ with analytic $(\theta,\epsilon_1)$-dependent coefficients. 
By Lemma \ref{FormalSeries1}, we obtain formal series solutions:
\begin{equation} \eqlab{Type2,2}
    (z_1,w_1)(\theta,r_1,\epsilon_1) = \sum_{n=0}^{\infty} \e^{\i(n+1)\theta} m_{n2}^{\psi}(\epsilon_1 \e^{-2\i\theta}) r_1^n,
\end{equation}
where $\psi\in \{\psi^S,\psi^U\}$. If we take 
\begin{align*}
    \psi = \psi^S = -2\theta^{S,E} = \arg(\lambda_0) - \frac{\pi}{2},
\end{align*}
recall \eqref{psiSU}, 
then the coefficients of \eqref{Type2,2} are well-defined for:
\begin{equation}\eqlab{s1def}
    \epsilon_1 \in [0,\delta), \quad \theta \in S^S:= \left\{ \theta \in \R / (2\pi \Z)\,:\,  \vert \theta - \theta^{S,E} \vert < \frac{\pi + \chi}{4}\right\}.
\end{equation} 
Recall that $\chi$ is any number in $(0,\pi)$. This follows from \eqref{chanco} and the fact that $\epsilon_2\in S(\psi^S,\chi,\delta)$. Due to doubling of the angle when going from $\theta$ to $\psi$, the coefficients of \eqref{Type2,2} are also well-defined when replacing $\theta^{S,E}$ with  $\theta^{S,E}+\pi$ in \eqref{s1def}. For simplicity, we will use the same symbol ($S^S$) for the resulting set. Similarly, if we take \begin{align*}
    \psi = \psi^U = -2\theta^{U,E} = \arg(\lambda_0) + \frac{\pi}{2},
\end{align*}
see \eqref{psiSU},
then the coefficients of \eqref{Type2,2} are well-defined for:
\begin{equation}\eqlab{s1defnew}
    \epsilon_1 \in [0,\delta), \quad \theta \in S^U:= \left\{ \theta \in \R / (2\pi \Z)\,:\,  \vert \theta - \theta^{U,E} \vert < \frac{\pi + \chi}{4}\right\}.
\end{equation}
Again the coefficients are also well-defined when taking $\theta^{U,E} \mapsto \theta^{U,E}+\pi$ in \eqref{s1defnew}. For simplicity, we will again use the same symbol ($S^U$) for the resulting set. Then by assumptions \ref{B1}-\ref{B2}, we can take $\chi\in (0,\pi)$ large enough such that $\theta^{\kC} \in S^S\cap S^U$, recall Figure \ref{Anglesab}. 

\section{Analytic continuation of the slow manifolds}\label{secextension}
We now turn to extending the slow manifolds into the complex plane. 
\subsection{Continuation to \texorpdfstring{$\mu = 0$}{!}}
Our first result concerns the extension to the non-hyperbolic point at the origin.  For this, we consider \eqref{VecSys}.

\begin{proposition}\label{Extension1}
    Fix $\delta >0$ small enough. Then the fixed slow manifolds can be analytically extended  up to $\mu=0$, such that:
    \begin{align*}
        & \xi = \xi^a(\mu,\epsilon), \quad \mu \in [\mu_s+\delta,0],\\
        & \xi = \xi^r(\mu,\epsilon), \quad \mu \in [0,\mu_u-\delta],
    \end{align*}
    for all $0 \leq \epsilon \ll 1$. In particular, $\xi^a(\mu,\epsilon)$, $\xi^r(\mu,\epsilon)$ are both asymptotic to the formal series \eqref{FS1} for all relevant $\mu$.
    
\end{proposition}

\begin{proof}
Lemma \ref{FormS1} provides the existence of a formal invariant manifold:
\begin{equation}\eqlab{TildeFS10}
    \zeta = \sum_{n=1}^{\infty}\xi_n(\mu)\epsilon^n,
\end{equation}
of \eqref{VecSys} 
with $\mu \in \Sigma_0$ and $\xi_n$ analytic for all $n \in \N$. Now, fix $N\in \N$ and define $\tilde{\zeta}$ by
\begin{equation}\eqlab{TildeFS1}
    \zeta = \sum_{n=1}^{2N-1} \xi_n(\mu)\epsilon^n + \epsilon^N \tilde \zeta
\end{equation}
As a result of Lemma \ref{FormS1}, the fixed slow manifolds $\xi^a(\mu,\epsilon)$ and $\xi^r(\mu,\epsilon)$ are both asymptotic to \eqref{TildeFS10}. Hence, in the coordinates $\tilde \zeta = (\tilde z, \tilde w)$ the slow manifolds take the form $\tilde \zeta = \epsilon^N\tilde \xi^a(\mu,\epsilon)$, $\mu \in [\mu_s-\delta,-\delta]$, and $\tilde \zeta = \epsilon^N\tilde \xi^r(\mu,\epsilon)$, $\mu \in [\delta,\mu_u+\delta]$, for all $ 0 \leq \epsilon \ll 1$ and $\delta > 0$ sufficiently small. 
The proof is then obtained by working in the $\tilde \zeta$ coordinates, and will be done for the attracting part only. The argument for the repelling case is similar (using the backward flow) and therefore left out.
 
Viewing \eqref{TildeFS1} as a change of coordinates, we now show that $\tilde \zeta = (\tilde z,\tilde w)$ solves a system of the form:
\begin{equation}\eqlab{TildeSystem}
        \begin{aligned}
        &\tilde \zeta' = (\operatorname{diag}(\lambda(\mu),\nu(\mu)) +\kO (\epsilon)) \tilde \zeta + \kO(\epsilon^N),\\
        &\mu'=\epsilon,
    \end{aligned}
\end{equation}
with all terms being analytic with respect to $\mu \in \Sigma_0$. For this purpose, let
\begin{equation}\nonumber
    W^N(\mu,\epsilon):  = \sum_{n=1}^{2N-1} \xi_n(\mu)\epsilon^n.
\end{equation}
Then $\zeta=W^N(\mu,\epsilon)$ defines an invariant manifold of \eqref{VecSys} up to order $\mathcal O(\epsilon^{2N})$.
In further details, there exists an analytic function $R(\mu,\epsilon)$, $\mu \in \Sigma_0$, such that:
\begin{equation}\eqlab{Expansion1}
    \begin{aligned}
     &\epsilon^{2N} R(\mu,\epsilon):=\operatorname{diag}(\lambda(\mu),\nu(\mu))W^N + \epsilon H(W^N,\mu,\epsilon)-\epsilon \frac{d}{d\mu}W^N ,
    \end{aligned}
\end{equation}
for all $0 \leq \epsilon \ll 1$. Using this, we obtain:
\begin{equation}\nonumber
    \epsilon^N \dot{\tilde \zeta} = \epsilon^N \operatorname{diag}(\lambda(\mu),\nu(\mu))\tilde \zeta + \operatorname{diag}(\lambda(\mu),\nu(\mu))W^N + \epsilon H(W^N+\epsilon^N\tilde \zeta,\mu,\epsilon)-\epsilon \frac{d}{d\mu}W^N,
\end{equation}
upon applying the change of coordinated defined by \eqref{TildeFS1}.
Now, expanding the last three terms and using \eqref{Expansion1}, we obtain the result of \eqref{TildeSystem}.
 
 Next, let $\mathcal{M}(t,\mu(0),\epsilon)$ denote the fundamental matrix with  $\mathcal{M}(0,\mu(0),\epsilon)=\mathbb I_2$ associated with the linear system \begin{equation}\tilde \zeta ' = (\operatorname{diag}(\lambda(\mu),\nu(\mu)) +\kO (\epsilon)) \tilde \zeta, \quad \mu(t) = \mu(0) + \epsilon t.\end{equation} Here we have fixed $\mu(0) \in [\mu_u-\delta,-\delta]$. 
 The system (\theequation) is the linear part of the system for the remainder $\tilde \zeta$ in \eqref{TildeSystem}.
  Moreover, let
\begin{equation}\label{TransitionTime41}
    T= \frac{\vert \mu(0) \vert}{\epsilon},
\end{equation}
denote the transition time
from $\mu = \mu(0)$ to $\mu = 0$. 
 Suppressing the dependency on $\mu(0)$ and $\epsilon$, then we have 
   %
    \begin{equation}\nonumber
        \norm{ \kM(t)\kM(s)^{-1}} \leq \e^{c_1\epsilon(t-s)} \leq \kO(1),
    \end{equation}
    since $\real(\lambda(\mu))\leq 0$ and $\real(\nu(\mu))\leq 0$ for $0\leq s \leq t\leq T$ by assumption \ref{A1}. To achieve this, one can e.g. use the fixed point formulation:
    \begin{align*}
    \kM(t)\kM(s)^{-1}  =&\exp\left({\operatorname{diag}\left(\int_s^t \lambda(\tau) d\tau,\int^t_s \nu(\tau) d\tau\right)}\right) \\
    &+ \int_s^t \exp\left({\operatorname{diag}\left(\int_s^u \lambda(\tau) d\tau,\int^u_s \nu(\tau) d\tau\right)}\right) \mathcal O(\epsilon) \kM(u)\kM(s)^{-1} du,  
    \end{align*}
    obtained by variation of constants and the fact that $$\norm{\exp\left({\operatorname{diag}\left(\int_s^t \lambda(\tau) d\tau,\int^t_s \nu(\tau) d\tau\right)}\right)}\le 1.$$
     
    We now solve  \eqref{TildeSystem} for $\tilde \zeta(t)$ through a fixed point formulation with initial condition $\tilde \zeta(0) = \epsilon^N \tilde \xi^a(\mu(0),\epsilon)$ in the following way
    \begin{equation}\nonumber
        \tilde \zeta(t) = \kM(t,\epsilon)\epsilon^N \tilde \xi^a(\mu(0),\epsilon) + \int_{0}^{t} \kM(t,\epsilon) \kM(s,\epsilon)^{-1}\kO(\epsilon^N)  ds,
    \end{equation}
    also obtained by variation of constants. 
    Assume that $|\tilde \zeta(t)|\le C$ for all $t\in [0,T]$. This then leads to the following estimate:
    \begin{equation}\eqlab{Ineq41}
        \vert \tilde \zeta(t) \vert \leq \kO(1)\epsilon^{N} + \kO(1)t \epsilon^{N}\leq \kO(1)\epsilon^{N-1},
    \end{equation}
    for $0<\epsilon\ll 1$ and all $0\leq t \leq T$. The last estimate was obtained using the explicit form of the transition time. We see that (\theequation) implies that $\tilde \zeta(t)$ is uniformly bounded for any $N \in \N$.
    From this we conclude that the attracting slow manifold can be extended up to the point $\mu=0$. Moreover, by the arbitrariness of $N$ it follows that the extended manifold is asymptotic to the formal series. This completes the proof.
\end{proof}
\subsection{Continuation to \texorpdfstring{$\kC_*$}{!} uniformly bounded away from \texorpdfstring{$\mu_*$}{!}}
Next, we turn to extending the slow manifold into the complex $\mu$-plane. Again, we consider $\Sigma_0$ (i.e. uniformly bounded away from $\mu_*$), but now we study the system \eqref{ELI} (in vector form):
\begin{equation}
        \begin{aligned}
        & \zeta' = \operatorname{diag}(-\i \vert\lambda(\mu)\vert^2,-\i \overline{\lambda(\mu)}\nu(\mu))  \zeta -\i \overline{\lambda(\mu)} \epsilon H(\zeta,\mu,\epsilon),\\
        &\mu'=-\i\epsilon\overline{\lambda(\mu)},
    \end{aligned}
\end{equation}
This system is equivalent to \eqref{VecSys} under the rescaling $- \i \overline{\lambda(\mu)}$ and hence the formal series of Lemma \ref{Extension1} also define formal invariant manifold solution of (\theequation).
We reintroduce $\tilde \zeta$ through a truncation of the series \eqref{TildeFS10}:
\begin{equation}\eqlab{TildeFS2}
    \zeta = \sum_{n=1}^{2N-1} \xi_n(\mu)\epsilon^n + \epsilon^N \tilde \zeta.
\end{equation}
Then by similar calculations as in the proof in Proposition \ref{Extension1}, we obtain the system:
\begin{equation}\eqlab{etildesys}
        \begin{aligned}
        &\tilde \zeta' = (\operatorname{diag}(-\i \vert\lambda(\mu)\vert^2,-\i \overline{\lambda(\mu)}\nu(\mu)) +\kO (\epsilon)) \tilde \zeta + \kO(\epsilon^N),\\
        &\mu'=-\i\epsilon\overline{\lambda(\mu)}.
    \end{aligned}
\end{equation}
\begin{proposition}\label{E4}
    Let $\kC_{0*}$ be any compact arc contained in the interior of $\kC_*$. Then the invariant manifolds can be uniformly extended  up to $\kC_{0*}$, such that 
    \begin{align*}
        & \zeta = \xi^a(\mu,\epsilon), \quad \mu \in \kC_{0*},\\
        & \zeta = \xi^r(\mu,\epsilon), \quad \mu \in \kC_{0*},
    \end{align*}
    for all $0 \leq \epsilon \ll 1$. In particular, $\xi^a(\mu,\epsilon)$, $\xi^r(\mu,\epsilon)$ are both asymptotic to the formal series \eqref{TildeFS10} for all relevant $\mu$.
\end{proposition}
\begin{proof}
We only provide a sketch of the proof, as it is similar to the proof of Proposition \ref{Extension1}. 
Consider the fundamental matrix $\mathcal{M}(t,\mu(0),\epsilon)$ with $\mathcal{M}(0,\mu(0),\epsilon)=\mathbb I_2$ associated with the linear system 
\begin{equation}
    \tilde \zeta' = (\operatorname{diag}(-\i \vert\lambda(\mu)\vert^2,-\i \overline{\lambda(\mu)}\nu(\mu)) +\kO (\epsilon)) \tilde \zeta,
\end{equation}
where $\mu = \mu(t)$ is a solution of $\mu' = -\i\epsilon \overline{\lambda(\mu)}$, $\mu(0) \in [\mu_s-\delta,0]$, $0<\delta$ fixed small enough. The system (\theequation) is the linear part of the system for the remainder $\tilde \zeta$. 
By assumption \ref{B2} and \ref{B3} (see also Figure \ref{EBox}), $\mu(T) \in \kC_{0*}$ for some smallest $0 \leq T=T(\mu(0),\epsilon)$ of $\kO(\epsilon^{-1})$. We now suppress the dependency on $\mu(0)$ and $\epsilon$ and proceed as in Proposition \ref{Extension1}: First, we use the fact that the real part of the diagonal matrix is non-positive to bound $\norm{\mathcal{M}(t)\mathcal M(s)^{-1}} = \kO(1)$ for all $0 \leq s \leq t \leq T$. 
Secondly, we consider a fixed point formulation for $\tilde \zeta$ which leads to the following estimate:
$\vert \tilde \zeta(t) \vert \leq \kO(1) \epsilon^{N-1}$ for all $0 \leq \epsilon \ll 1$. By assumption (B2) the result follows.

\end{proof}

\subsection{Continuation to \texorpdfstring{$\kC_*$}{!} near the root \texorpdfstring{$\mu_*$}{!}} 
In this section, we consider the elliptic system  in the $\vert \breve \mu \vert = 1$-chart, see \eqref{bchart2}, repeated here for convenience:
\begin{equation}\eqlab{bchart2extra}
\begin{aligned}
    &z_1'=r_1[-\i\vert\hat \lambda\vert^2z_1-\i\e^{-\i\theta}\overline{\hat \lambda}\epsilon_1F-\epsilon_1z_1\real[\i\hat \lambda\e^{\i\theta}]],\\
    &w_1' = -\i\e^{-\i\theta}\overline{\hat \lambda}\nu w_1-\i\e^{-\i\theta}\overline{\hat \lambda}r_1\epsilon_1G-r_1\epsilon_1w_1\real[\i\hat \lambda\e^{\i\theta}],\\
    &\theta' =\epsilon_1r_1\real[-\hat \lambda\e^{\i\theta}],\\
    &r_1 '= r_1^2\epsilon_1\real[\i\hat \lambda\e^{\i\theta}],\\
    &\epsilon_1'= -2\epsilon_1^2r_1\real[\i\hat \lambda\e^{\i\theta}].
\end{aligned}
\end{equation}
First, we consider the truncation of the formal invariant manifold solution in \eqref{Type2,2}: 
\begin{equation}\nonumber
    (z_1,w_1)(\theta,r_1,\epsilon_1) = \sum_{n=0}^{L} \e^{\i(n+1)\theta}m_{n2}^{\psi}(\epsilon_1 \e^{-2\i\theta}) r_1^n, \quad L \in \N,
\end{equation}
with $\psi \in \{\psi^S,\psi^U\}$. This finite sum defines an invariant manifold solution of \eqref{bchart2extra} up to $\mathcal{O}(r_1^{L+1})$ for all $L\in \N$. Similarly, by truncating the type I formal series from Lemma \ref{FormalSeries1} (given by equation \eqref{FRØ1}) 
\begin{equation}\nonumber
    (z_1,w_1)(\theta,r_1,\epsilon_1) = \sum_{n=1}^{M} \e^{\i(-2n+1)\theta}\xi_{n2}(r_1 \e^{\i\theta}) \epsilon_1^n, \quad M \in \N,
\end{equation}
we obtain an invariant manifold solution up to $\kO(\epsilon_1^{M+1})$ for all $M \in \N$. Let $J^\alpha[f], \alpha \in \N$ denote the $\alpha$-jet acting on an analytic function $f$ 
\begin{equation}\nonumber
    J^\alpha[x \mapsto \sum_{n=0}^{\infty} f_nx^n ](x) := \sum_{n=0}^{\alpha} f_nx^n
\end{equation}
We then combine the two formal invariant manifolds (taking $L=2N-1$ and $M=4N-1$) in the following way 
 \begin{equation}\eqlab{DoubleSeries}
    \begin{aligned}
        W_1^N(\theta,r_1,\epsilon_1) &:=  \sum_{n=1}^{2N-1}\e^{\i(-2n+1)\theta}\xi_{n2}(r_1 \e^{\i\theta}) \epsilon_1^n\\
        &+\sum_{n=0}^{4N-1}\e^{i(n+1)\theta}(m^{\psi}_{n2}(\e^{-2i\theta}\epsilon_1)-J^{2N-1}[\epsilon_1 \mapsto m^{\psi}_{n2}(\e^{-2i\theta}\epsilon_1)](\epsilon_1))r_1^n.
    \end{aligned}
\end{equation}
Notice that the terms removed in the second sum by the jet are already accounted for in the first sum, cf. Lemma \ref{Doublesseriescof}. Hence, $(z_1,w_1) = W_1^N(\theta,r_1,\epsilon_1)$ defines an invariant manifold up to terms of $\kO(\epsilon_1^{2N})$ as well as an invariant manifold up to terms of $\kO((r_1^2)^{2N})$.  In conclusion, \eqref{DoubleSeries} is an invariant manifold up to $\kO((r_1^2\epsilon_1)^{2N})$.  More precisely, there is an analytic function $Q(\theta,r_1,\epsilon_1)$ such that 
\begin{equation}\eqlab{Qdef}
    \begin{aligned}
    &(r_1^2\epsilon_1)^{2N}Q(\theta,r_1,\epsilon_1) :=\\
    & -\partial_{\theta}(W^N_1) \epsilon_1r_1\real[-\hat \lambda \e^{\i\theta}]-\partial_{r_1}(W^N_1)\epsilon_1r_1^2\real[\i\hat \lambda \e^{\i\theta}]+2\partial_{\epsilon_1}(W^N_1)\epsilon_1^2r_1\real[\i\hat \lambda \e^{\i\theta}] \\
        & +P(\theta,r_1) W^N_1-\i\e^{-\i\theta}\overline{\lambda}_0r_1\epsilon_1H-r_1\epsilon_1W^N_1\real[\i\hat \lambda \e^{\i\theta}],
    \end{aligned}
\end{equation}
recall \eqref{Zeta1inv}. This procedure for combining different asymptotic series on the blowup sphere follows \cite[Lemma 4.8]{kriszm1}. In the following we fix $\psi = \psi^S$. Then \eqref{Qdef} is well-defined for $\theta \in S^S$, recall \eqref{s1def}.
\begin{proposition}\label{Coordinates2}
    Fix $N \in \N$, $N \gg 1$. Then the transformation 
    \begin{equation}\nonumber
        (\tilde{\zeta}_1,\theta,r_1,\epsilon_1) \mapsto (\zeta_1,\theta,r_1,\epsilon_1), 
    \end{equation}
    where $\zeta_1 = (z_1,w_1)$ and $r_1,\epsilon_1 \in [0,\delta), \theta \in S^S$, defined by
    \begin{equation}\nonumber
        \zeta_1 = W_1^N(\theta,r_1,\epsilon_1)+(r_1^2\epsilon_1)^N\tilde{\zeta}_1,
    \end{equation}
    brings the system \eqref{bchart2extra} into the following form 
    \begin{equation}\eqlab{r1e1system}
        \begin{aligned}
        & \tilde\zeta_1' =  (P(\theta,r_1)+ r_1\epsilon_1 H^N(\theta,r_1,\epsilon_1))\tilde{\zeta}_1+(\epsilon_1r_1^2)^NR^N(\tilde \zeta_1,\theta,r_1,\epsilon_1),\\
        &\theta' =\epsilon_1r_1\real[-\hat \lambda \e^{\i\theta}],\\
        &r_1 '= r_1^2\epsilon_1\real[\i\hat \lambda \e^{\i\theta}],\\
        &\epsilon_1'= -2\epsilon_1^2r_1\real[\i\hat \lambda \e^{\i\theta}],
        \end{aligned}
    \end{equation}
    with 
    \begin{equation}\eqlab{Hn}
    H^N(\theta,r_1,\epsilon_1) = -\operatorname{Re} [\i \lambda_0\e^{\i\theta}] \mathbb{I}_{2}+\kO(r_1).
    \end{equation}
\end{proposition}
\begin{proof}
    We compute 
    \begin{equation}\nonumber
       \begin{aligned}
            (r_1^2\epsilon_1)^{N}\tilde \zeta_1' &=\zeta_1'  -P(\theta,r_1) W^N_1+\i\e^{-\i\theta}\overline{\lambda}_0r_1\epsilon_1H\\
            &+r_1\epsilon_1W^N_1\real[\i\hat \lambda \e^{\i\theta}]+ (r_1^2\epsilon_1)^{2N}Q(\theta,r_1,\epsilon_1)\\
            & = (r_1^2\epsilon_1)^{N}P(\theta,r_1) \tilde\zeta_1-(r_1^2\epsilon_1)^{N}\i\e^{-\i\theta}\overline{\lambda}_0r_1^2\epsilon_1 \widetilde H \tilde \zeta_1\\
            &-(r_1^2\epsilon_1)^{N}r_1\epsilon_1\tilde \zeta_1\real[\i\hat \lambda \e^{\i\theta}] +(r_1^2\epsilon_1)^{2N}Q(\theta,r_1,\epsilon_1),
       \end{aligned}
    \end{equation}
    using \eqref{Qdef} and the mean value theorem in the last equality:
    \begin{equation}\nonumber
    \resizebox{0.99\linewidth}{!}{$
        r_1(r_1^2\epsilon_1)^{N} \widetilde H((r_1^2\epsilon_1)^{N}\tilde \zeta_1,\theta,r_1,\epsilon_1) \tilde \zeta_1 := H(r_1\zeta_1,\mu_*+r_1\e^{\i\theta},r_1^2\epsilon_1) - H(r_1W^N_1,\mu_*+r_1\e^{\i\theta},r_1^2\epsilon_1).
        $}
    \end{equation}
    The statement now follows from a division by $(r_1^2\epsilon_1)^N$.
\end{proof}
To proceed, we now consider the $(\theta,r_1,\epsilon_1)$-system 
\begin{equation}\eqlab{øre}
    \begin{aligned}
    &\theta' =\epsilon_1r_1\real[-\hat \lambda \e^{\i\theta}],\\
        &r_1 '= r_1^2\epsilon_1\real[\i\hat \lambda \e^{\i\theta}],\\
        &\epsilon_1'= -2\epsilon_1^2r_1\real[\i\hat \lambda \e^{\i\theta}],
        \end{aligned}
\end{equation}which decouples, and work in the following domain $(\theta,r_1,\epsilon_1) \in S^S \times [0,\delta]\times[0,\delta]$ with $\delta>0$ small enough. We then define the section 
\begin{equation}\nonumber
    \Sigma_{in}=\{ (\theta,r_1,\epsilon_1) \in S^S \times [0,\delta]\times[0,\delta]\,:\, r_1 = \delta \}.
\end{equation}
Moreover, let $\kC_{1*} $ denote the blowup of the set defined by $(\mu, \epsilon) \in \kC_* \times \{0 \leq \epsilon \leq \epsilon_0\}$, with $0 < \epsilon_0 \ll 1$, in the $(\theta,r_1,\epsilon_1)$-coordinates. Then by assumption \ref{B1}-\ref{B2}, $\kC_{1*}$ intersects our domain as  a graph over $(r_1,\epsilon_1)$ near $\theta = \theta^{\kC} \in (\theta^{S,E},\theta^{U,E})$ or $\theta = \theta^{\kC} \in (\theta^{S,E}+\pi,\theta^{U,E}+\pi)$, recall Figure \ref{Anglesa} (for simplicity, we henceforth focus on the former).  See also Figure \ref{BlowupFigur}. In particular, $\kC_{1*}\cap \{r_1=0\}$ is given by $\theta=\theta^{\kC}$, $\epsilon_1\ge 0$. This should be clear enough. 

Now, by proceeding as in the proof of Proposition \ref{E4} (since we are uniformly bounded away from $\mu_*$), we obtain an extension of the attractive slow manifold as a graph over $(\theta,r_1,\epsilon_1) \in \Sigma_{in}$. We now wish to extend this to $\kC_{1*}$.
\begin{figure}
    \centering
    \includegraphics[width=0.8\linewidth]{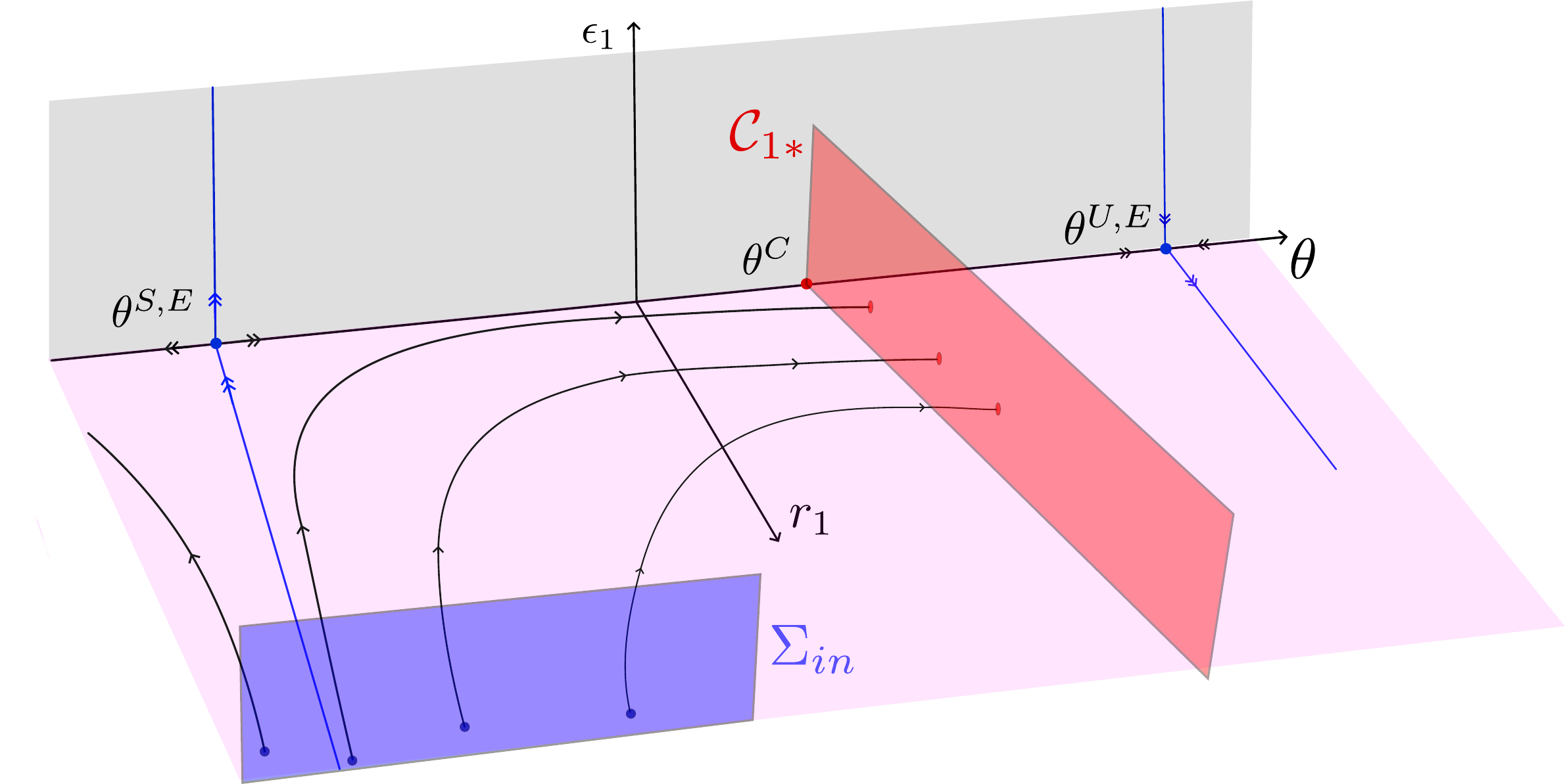}
    \caption{The dynamics of the elliptic system in the $(\theta,r_1,\epsilon_1)$-coordinates. The angles $\theta^{S,E}$ and $\theta^{U,E}$ are the angles associated with the tangents of the stable and unstable manifolds at $\mu_*$, recall Lemma \ref{Lemma21}, whereas $\theta^{\kC}$ is the angle associated with the tangent of $\kC_*$ at $\mu_*$, see also Figure \ref{Anglesab}. The set $\kC_{1*}$ denotes the blowup of $\kC_*\times \{0\le \epsilon<\epsilon_0\}$ in the $(\theta,r_1,\epsilon_1)$-coordinates.} 
    \label{BlowupFigur}
\end{figure}
Since $(\theta,r_1,\epsilon_1) = (\theta^{S,E},0,0)$ is a hyperbolic saddle (this follows from Lemma \ref{Lemma21}), and the flow is transverse to $\kC_{1*}$, all points in $\kC_{1*}$ with $r_1>0$ reach $\Sigma_{in}$ under application of the backward flow, see Figure \ref{BlowupFigur}. We define the subset $\widetilde \Sigma_{in} \subset \Sigma_{in}$ as the resulting set of points. 
Now, as a consequence of the uniqueness of the formal series (see Lemma \ref{FSU3}), we conclude that the attracting slow manifold takes the following graph form on $\widetilde \Sigma_{in}$:
\begin{equation}\eqlab{InCond}
    \tilde \zeta_1 = \mathcal{O}((r_1^2\epsilon_1)^N), \quad (\theta,r_1,\epsilon_1) \in \widetilde \Sigma_{in},
\end{equation}
using the $(\tilde \zeta_1,\theta,r_1,\epsilon_1)$-coordinates of Proposition \ref{Coordinates2}. We will use the set \eqref{InCond} as initial conditions for the flow of \eqref{r1e1system} and extend the attracting slow manifold as a graph over $\kC_{1*}$.

Let $(\theta(t),r_1(t),\epsilon_1(t))$ with $ (\theta(0),r_1(0),\epsilon_1(0)) \in \widetilde \Sigma_{in}$ denote a forward solution of \eqref{øre}. At the same time, we define $T(\theta(0),r_1(0),\epsilon_1(0))>0$ as the transition time from $\widetilde \Sigma_{in}$ to $\kC_{1*}$. Then we have the following result:
\begin{lemma}\label{TransitionTIme}
    There exists a constant $K>0$ such that 
    \begin{equation}\nonumber
        T(\theta(0),r_1(0),\epsilon_1(0))\leq K\epsilon_1(0)^{-1}, \quad \forall\, (\theta(0),r_1(0),\epsilon_1(0)) \in \widetilde \Sigma_{in}.
    \end{equation}
\end{lemma}
\begin{proof}
    We eliminate $\epsilon_1$ by the conversation $\epsilon = r_1^2\epsilon_1$.
    Let $$B = \min \left\vert \real[-\hat \lambda \e^{\i\theta}] \right\vert,$$ with $\min$ taking over the subset of the domain $S^S \times [0,\delta]\times [0,\delta]$ with $\theta$ strictly bounded away from $\theta^{S,E}$. By assumption $B$ is strictly positive on this domain, see Figure \ref{BlowupFigur}. We therefore have
    \begin{equation}\nonumber
        \theta' \geq r_1^{-1}\epsilon B \geq \delta^{-1} \epsilon B,
    \end{equation}
    using $r_1(t) \leq \delta$ in the second inequality. The bound on $T$ follows for this subset of initial conditions upon using that $\epsilon_1(0) = \frac{\epsilon}{\delta^2}$.

    Now we consider initial conditions close to $\theta^{S,E}$. Suppose that $\theta(t)$ remains uniformly close to $\theta^{S,E}$ for $t \in [0,T_0]$ for some $T_0 < T$. In this time period, $\epsilon_1$ (which is real) is strictly increasing:
    \begin{equation}\nonumber
        \epsilon_1' \geq 2 B \epsilon_1^{\frac{3}{2}}\sqrt{\epsilon},
    \end{equation}
    where now $B = \min \vert \real[-\i\overline \lambda_0 \e^{-\i\theta}]\vert$ with the minimum taking over the domain with $\theta$ close to $\theta^{S,E}$. By simple integration, this leads to the estimate: 
    \begin{equation}\nonumber
        T_0 \leq \frac{\delta}{B\epsilon}.
    \end{equation}
    The point $(\theta(T_0),r_1(T_0),\epsilon_1(T_0))$ is then an initial condition uniformly bounded away from $\theta^{S,E}$. The result therefore follows from the first part of the proof.
\end{proof}

With $(\theta(t),r_1(t),\epsilon_1(t))$ as above, we now turn to the linear equation
\begin{equation}\eqlab{Ls1}
    \begin{aligned}
        \tilde\zeta_1' &=  (P(\theta,r_1)+ r_1\epsilon_1 H^N(\theta,r_1,\epsilon_1))\tilde{\zeta}_1\\
        &=(P(\theta,r_1)- r_1\epsilon_1 \operatorname{Re} [\i \hat \lambda\e^{\i\theta}] \mathbb{I}_{2}+\kO(\epsilon))\tilde{\zeta}_1\\
        & = (P(\theta,r_1)-r_1'r_1^{-1} \mathbb{I}_{2}+\kO(\epsilon))\tilde{\zeta}_1
    \end{aligned}
\end{equation}
using \eqref{Hn} and $r_1^2\epsilon_1=\epsilon$ in the second equality and \eqref{øre} in the third equality. Notice, that the components of $P$ have non-positive real part since we are on the attracting side of $\kC_{1*}$, recall \eqref{P}. In particular, the first component is imaginary, while the second component has negative real part on the relevant side of $\kC_{1*}$.
 
Let $\mathcal{M}(t;(\theta(0),r_1(0),\epsilon_1(0)))$ be the fundamental matrix $\widetilde{\kM}$ with $$\mathcal{M}(0;(\theta(0),r_1(0),\epsilon_1(0)) = \mathbb{I}_2,$$ associated with the linear system \eqref{Ls1} for $(\theta(0),r_1(0),\epsilon_1(0)) \in \widetilde \Sigma_{in}$. Here and henceforth we suppress the dependency on $(\theta(0),r_1(0),\epsilon_1(0))$.
\begin{lemma}\label{lemma45}
    The following bound holds 
    \begin{equation}\eqlab{SBound}
        \norm{\mathcal{M}(t)\mathcal{M}(s)^{-1}} \leq c\epsilon_1(0)^{-1},
    \end{equation}
    for all $0 \leq s \leq t \leq T$ with $c>0$ large enough and independent of $(\theta(0),r_1(0),\epsilon_1(0))$.
\end{lemma}
\begin{proof}
    Define $\widetilde{\mathcal{M}}$ by
    \begin{equation}\nonumber
        \mathcal{M}(t) = \frac{r_1(0)}{r_1(t)} 
        \widetilde{\mathcal{M}}(t).
    \end{equation}
    Then $\widetilde{\mathcal{M}}(t)$ is the fundamental matrix with $\widetilde{\mathcal{M}}(0) = \mathbb{I}_2$ associated with the problem 
    \begin{equation}\nonumber
        \tilde{\zeta}_1'=(P(\theta,r_1)+\kO(\epsilon))\tilde{\zeta}_1,
    \end{equation}
i.e. we have removed the second term (proportional to the identity $\mathbb I_2$) from the coefficient matrix in \eqref{Ls1}. 
Now, upon using that the components of $P$ have non-positive real part and that $T \leq K \epsilon^{-1}$ for some $c_1>0$ large enough independent of $(\theta(0),r_1(0),\epsilon_1(0))$, we have (as in Proposition \ref{Extension1}) that
%
 %
    \begin{equation}
\nonumber\norm{\widetilde{\mathcal{M}}(t)\widetilde{\mathcal{M}}(s)^{-1}} \leq c_1,
    \end{equation}
    for all $0 \leq s \leq t \leq T$. 
 The result then follows from estimating: 
    \begin{equation}\nonumber
        \norm{\mathcal{M}(t)\mathcal{M}(s)^{-1}}= \norm{\widetilde{\mathcal{M}}(t)\widetilde{\mathcal{M}}(s)^{-1}} \frac{ r_1(s)}{ r_1(t)}   
        \leq c_1\frac{\delta^2}{\epsilon}.
    \end{equation}
    Here we have used that $\delta^{-1}\epsilon \leq r_1(t) \leq \delta$ for all $t \in [0,T]$.
\end{proof}
We can now write the following fixed point formulation:
\begin{equation}\nonumber
    \begin{aligned}
        \widetilde{\zeta}_1(t) &= \mathcal{M}(t)\widetilde \zeta_1(0)  \\ &+\int_{0}^{t}\mathcal{M}(t)\mathcal{M}(s)^{-1}\epsilon^NR^N(\widetilde{\zeta}_1(s),\theta(s),r_1(s),\epsilon_1(s))ds.
    \end{aligned}
\end{equation}
using variation of constants, 
for the solution $\widetilde{\zeta}_1(t)$ of \eqref{r1e1system}  with $\widetilde \zeta_1(0) = \kO(\epsilon^N)$, recall \eqref{InCond}. Here we have used that  $r_1^2\epsilon_1 = \epsilon$, see \eqref{chart2}. Suppose that 
    \begin{equation}\nonumber
        \vert \widetilde \zeta_1(t) \vert \leq C,
    \end{equation}
    for all $t \in [0,T]$. Then using \eqref{SBound}, we obtain that
    \begin{equation}\nonumber
        \vert \widetilde{\zeta}_1(t) \vert \leq  \epsilon^{-1}\kO(\epsilon^N) + T \kO(\epsilon^{N-1}) \leq \kO(\epsilon^{N-2}), \quad t \in [0,T],
    \end{equation}
    being uniformly bounded for any $N \geq 3$. This directly leads to the following result: 
\begin{proposition}\label{F4Res}
    The attracting slow manifold extends as an invariant graph $$\zeta_1=\xi^a_1(\theta,r_1,\epsilon_1),$$ of \eqref{bchart2} over $(\theta,r_1,\epsilon_1) \in \kC_{1*}$. Moreover, $\xi_1^a$ is asymptotic to the series \eqref{FRØ1} and \eqref{Type2,2} with $\psi = \psi^S$.
\end{proposition}

 \subsection{The splitting near \texorpdfstring{$\mu_*$}{!}}\seclab{secsplit}

As a result of Proposition \ref{F4Res}, we have obtained the attracting slow manifold along $\kC_{1*}$ as an invariant graph $\zeta_1=\xi_1^a(\theta,r_1,\epsilon_1)$ asymptotic to the series \eqref{Type2,2} defined for $\theta\in S^S$, $r_1,\epsilon_1\in [0,\delta]$. Analogously, we have an extension (by using the backward flow) of the repelling slow manifold along $\kC_{1*}$ as an invariant graph $\zeta_1=\xi_1^r(\theta,r_1,\epsilon_1)$ asymptotic to the series \eqref{Type2,2} defined for $\theta\in S^U$. Recall that $$\theta^{\kC}\in S^S \cap S^U.$$

We now fix $\epsilon_1=\delta>0$ small enough. Then $r_1 = \sqrt{\epsilon \delta^{-1}}$ by the conservation of $\epsilon$, see \eqref{chart2}. In particular, $\kC_{1*}\cap \{\epsilon_1=\delta\}$ is a smooth curve parametrized by $\sqrt{\epsilon}$:  
\begin{equation}\eqlab{theta123}
    \theta=\hat \theta(\sqrt{\epsilon}), \quad r_1=\sqrt{\epsilon\delta^{-1}}, \quad \epsilon_1=\delta,
\end{equation}
where $\hat \theta$ denotes a locally defined smooth function with $\hat \theta(0)=\theta^{\kC}$. 
Blowing back down using \eqref{chart2}, we find that $\kC_{1*}\cap \{\epsilon_1=\delta\}$ corresponds to 
\begin{align}
\mu = \mu_* + \sqrt{\epsilon} \hat \mu(\sqrt{\epsilon}),\quad \hat \mu(0) = \frac{1}{\sqrt{\delta}}\e^{i\theta^{\kC}}.\eqlab{mufinal}
\end{align}
with $$\hat \mu(\sqrt{\epsilon}):=\sqrt{\delta^{-1}}\e^{i\hat \theta(\sqrt{\epsilon})},$$ being a smooth complex-valued function. 
In the following, we define
\begin{equation}\eqlab{ell}
    \ell = 
    \begin{cases}
        0 \quad \text{if} \quad \theta^{\kC} \in (\theta^{S,E},\theta^{U,E}),\\
        1 \quad \text{if} \quad \theta^{\kC} \in (\theta^{S,E}+\pi,\theta^{U,E}+\pi),
    \end{cases}
\end{equation}
\begin{proposition}\label{finalprop} Fix $\delta>0$ small enough. Then the attracting and repelling slow manifolds $\xi^{a,r}=(z^{a,r},w^{a,r})$ extend to 
\begin{equation}\eqlab{muval}
\mu = \mu_* + \sqrt{\epsilon} \hat \mu(\sqrt{\epsilon}),
\end{equation} for all $0<\epsilon\ll 1$. In particular, for this value of $\mu$ we have 
\begin{equation}\eqlab{Deltatau}
\begin{aligned}
    z^{a}(\mu,\epsilon)-z^r(\mu,\epsilon)& = \sqrt{2\pi\epsilon}\left((-1)^{\ell-1} \frac{F_0}{\sqrt{\vert \lambda_0 \vert}}\e^{-\frac{3\i}{2} \arg(\lambda_0)} \e^{\frac{\lambda_0}{2\delta}\e^{2\i\theta^{\kC}}}  + \mathcal O(\sqrt{\epsilon})\right),\\
    w^{a}(\mu,\epsilon)-w^r(\mu,\epsilon) &= \mathcal O(\epsilon),
\end{aligned}
\end{equation}
for all $0 < \epsilon\ll 1$.
\end{proposition}
\begin{proof}
For $(\theta,r_1,\epsilon_1)\in  \kC_{1*}$, we have now established that 
\begin{align*}
    \xi^{a}_1(\theta,r_1,\epsilon_1) = \e^{i\theta} m_0^{\psi^S}(\epsilon_1 \e^{-2i\theta}) + \mathcal O(r_1),
\end{align*}
and similarly
\begin{align*}
    \xi^{r}_1(\theta,r_1,\epsilon_1) = \e^{i\theta} m_0^{\psi^U}(\epsilon_1 \e^{-2i\theta}) + \mathcal O(r_1).
\end{align*}
But then by using \eqref{chart2} (blowing down), we conclude that 
\begin{align}
    \xi^{a}(\mu,\epsilon) - \xi^r(\mu,\epsilon) &= r_1 (\xi_1^{a}(\theta,r_1,\epsilon_1) - \xi_1^{r}(\theta,r_1,\epsilon_1) )\nonumber \\
    &=r_1 \e^{i\theta} \left(m_0^{\psi^S}(\epsilon_1 \e^{-2i\theta})-m_0^{\psi^U}(\epsilon_1\e^{-2i\theta}) + \mathcal O(r_1)\right).\eqlab{deltazw}
\end{align}
Following Lemma \ref{SN1}, we now focus on the $z$-component of \eqref{deltazw} with $\epsilon_1 = \delta$, $r_1 = \sqrt{\epsilon\delta^{-1}}$ fixed (corresponding to \eqref{muval}):
\begin{align*}
&r_1 \e^{i\theta} \left(z_{02}^{\psi^S}(\epsilon_1 \e^{-2i\theta})-z_{02}^{\psi^U}(\epsilon_1\e^{-2i\theta}) + \mathcal O(r_1)\right)\\
    &=-\sqrt{\epsilon\delta^{-1}} \e^{\i \theta^{\kC}} \frac{\i F_0}{\sqrt{\vert \lambda_0 \vert}}\e^{-\i \arg(\lambda_0)} \e^{\frac{\lambda_0}{2\delta}\e^{2\i\theta^{\kC}}} \sqrt{2\pi\delta \e^{-\i(2\theta^{\kC}+ \arg(\lambda_0)+\pi)}} + \kO(\epsilon),
\end{align*}
see also \eqref{z02w02}.
Here we have used \eqref{theta123} and $\hat \theta(0) = \theta^{\kC}$. Now, suppose first that 
\begin{equation}
    \theta^{\kC} \in ( \theta^{S,E},\theta^{U,E}).
\end{equation}
Then the argument in the square root satisfies the inclusion:
\begin{equation}
    -2\theta^{\kC} - \arg(\lambda_0) - \pi \in (-\pi/2,\pi/2),
\end{equation}
i.e. it lies in the principle domain $(-\pi,\pi]$.
This follows directly from \eqref{thetasu}.
In this case, we can therefore evaluate the square root directly: 
\begin{align*}
    &-\sqrt{\epsilon\delta^{-1}} \e^{\i \theta^{\kC}} \frac{\i F_0}{\sqrt{\vert \lambda_0 \vert}}\e^{-\i \arg(\lambda_0)} \e^{\frac{\lambda_0}{2\delta}\e^{2\i\theta^{\kC}}} \sqrt{2\pi\delta \e^{-\i(2\theta^{\kC}+ \arg(\lambda_0)+\pi)}} \\
     &=-\sqrt{\epsilon\delta^{-1}} \e^{\i \theta^{\kC}} \frac{\i F_0}{\sqrt{\vert \lambda_0 \vert}}\e^{-\i \arg(\lambda_0)} \e^{\frac{\lambda_0}{2\delta}\e^{2\i\theta^{\kC}}} \sqrt{2\pi\delta} \e^{-\i(\theta^{\kC}+ \frac12 \arg(\lambda_0)+\frac12 \pi)} \\
    & = -\sqrt{(2\pi \epsilon)} \frac{F_0}{\sqrt{\vert \lambda_0 \vert}}\e^{-\frac{3\i}{2} \arg(\lambda_0)} \e^{\frac{\lambda_0}{2\delta}\e^{2\i\theta^{\kC}}} 
\end{align*}
Next, in the case where 
\begin{equation*}
    \theta^{\kC} \in ( \theta^{S,E}+\pi,\theta^{U,E}+\pi),
\end{equation*}
we find
\begin{equation*}
     -2\theta^{\kC} - \arg(\lambda_0) - \pi \in (-5\pi/2,-3\pi/2).
\end{equation*}
We therefore add a multiple of $2\pi$ to the argument of the square root (in order to move it into the principle domain $(-\pi,\pi]$). This leads to the sign change:
\begin{align*}
&-\sqrt{\epsilon\delta^{-1}} \e^{\i \theta^{\kC}} \frac{\i F_0}{\sqrt{\vert \lambda_0 \vert}}\e^{-\i \arg(\lambda_0)} \e^{\frac{\lambda_0}{2\delta}\e^{2\i\theta^{\kC}}} \sqrt{2\pi\delta \e^{-\i(2\theta^{\kC}+ \arg(\lambda_0)+\pi)}}\\
  &=  \sqrt{(2\pi \epsilon)} \frac{F_0}{\sqrt{\vert \lambda_0 \vert}}\e^{-\frac{3\i}{2} \arg(\lambda_0)} \e^{\frac{\lambda_0}{2\delta}\e^{2\i\theta^{\kC}}}.
\end{align*}
This completes the proof upon using \eqref{ell}.
\end{proof}
\section{Computing the difference }\label{secdiff}
Let $\widehat{\kC}_*(\epsilon)\subset \kC_*$ denote the closed subset of $\kC_*$ from $\mu=\overline{\mu_*}+\sqrt{\epsilon}\overline{\hat \mu}(\sqrt{\epsilon})$ to $\mu={\mu_*}+\sqrt{\epsilon}{\hat \mu}(\sqrt{\epsilon})$. Notice that $\widehat{\kC}_*(\epsilon)\to \kC_*$ as $\epsilon \to 0$ in the Hausdorff distance.
Then from the results of the previous section and Lemma \ref{lemmasym}, we have slow manifolds of \eqref{ComplexI} as graphs over $\widehat{\kC}_*$ 
\begin{equation}\eqlab{xiarfinal}
    (z,w) = \xi^{a,r}(\mu,\epsilon), \quad \mu \in \widehat{\kC}_{*}.
\end{equation}
We emphasize that $\xi^{a,r}(\mu,\epsilon)=\mathcal O(\epsilon)$ for $\mu$ bounded uniformly away from $\overline{\mu_*}$ and $\mu_*$. However, near these points we have $\xi^{a,r}(\mu,\epsilon) = \mathcal O(\sqrt{\epsilon})$, recall Proposition \ref{finalprop}.

We now consider the difference
\begin{equation}\nonumber
    \Delta = \xi^a(\mu,\epsilon)-\xi^r(\mu,\epsilon), \quad \Delta = (\Delta_z,\Delta_w).
\end{equation}
Seeing that $\xi^{a,r}=(z^{a,r},w^{a,r})$ are both invariant manifold solutions, we directly obtain the following equation for $\Delta$ 
\begin{align}\nonumber
    & \epsilon\frac{d\Delta}{d\mu} = 
    \begin{pmatrix}
        \lambda(\mu)&0\\
        0 & \nu(\mu)
    \end{pmatrix}
    \Delta + \epsilon
    \begin{pmatrix}
        F(z^a,w^a,\mu,\epsilon)-F(z^r,w^r,\mu,\epsilon)\\
        G(z^a,w^a,\mu,\epsilon)-G(z^r,w^r,\mu,\epsilon)
    \end{pmatrix}.
\end{align}
Moreover, as $F$ and $G$ are analytic, we find
    \begin{equation}\nonumber
         \begin{pmatrix}
         F(z^a,w^a,\mu,\epsilon)-F(z^r,w^r,\mu,\epsilon)\\
         G(z^a,w^a,\mu,\epsilon)-G(z^r,w^r,\mu,\epsilon)
     \end{pmatrix} =\begin{pmatrix}
         \partial_zF(z_b,w_b,\mu,\epsilon)&\partial_wF(z_b,w_b,\mu,\epsilon)\\
         \partial_zG(z_b,w_b,\mu,\epsilon) & \partial_wG(z_b,w_b,\mu,\epsilon)
     \end{pmatrix}
    \Delta,
    \end{equation}
by an application of the mean value theorem with $(z_b,w_b)=(z_b,w_b)(\mu,\epsilon)$ denoting the intermediary points. This leads to the linear equation for $\Delta(\mu,\epsilon)$ 
\begin{align}\eqlab{Delta1}
    & \epsilon\frac{d\Delta}{d\mu} = 
   \begin{pmatrix}
        \lambda(\mu)&0\\
        0 & \nu(\mu)
    \end{pmatrix}
    \Delta + \epsilon
    \begin{pmatrix}
        \partial_zF(z_b,w_b,\mu,\epsilon)&\partial_wF(z_b,w_b,\mu,\epsilon)\\
        \partial_zG(z_b,w_b,\mu,\epsilon) & \partial_wG(z_b,w_b,\mu,\epsilon)
    \end{pmatrix}
    \Delta.
\end{align}
To study this equation, we let $g:[-1,1]\rightarrow \kC_*$ be a $C^\infty$-smooth and regular parametrization of ${\kC}_*$ such that
\begin{equation}\eqlab{CSYM}
    g(0)=0, \quad g(1)=\mu_*, \quad g(-\tau) = \overline{g(\tau)}, 
\end{equation}
for all $\tau \in [0,1]$ which exists by assumption \ref{B1}. We let $\tau_*(\sqrt{\epsilon})$, $\tau_*(0)=1$, be such that $g(\tau_*)=\mu_*+\sqrt{\epsilon}\hat \mu(\sqrt{\epsilon})$. The existence and smoothness of $\tau_*(\sqrt{\epsilon})$ follows from the regularity of $g$. Consider now equation \eqref{Delta1} along the contour $\kC_*$ parametrized by $g(\tau)$:
\begin{equation}
\resizebox{0.9\linewidth}{!}{$
    \begin{aligned}\eqlab{gdsys}
    & \Delta'= 
   g'(\tau)\begin{pmatrix}
        \lambda(g(\tau))&0\\
        0 & \nu(g(\tau))
    \end{pmatrix}
    \Delta + \epsilon g'(\tau)
    \begin{pmatrix}
        \partial_zF(z_b,w_b,g(\tau),\epsilon)&\partial_wF(z_b,w_b,g(\tau),\epsilon)\\
        \partial_zG(z_b,w_b,g(\tau),\epsilon) & \partial_wG(z_b,w_b,g(\tau),\epsilon)
    \end{pmatrix}
    \Delta,\\
    &\tau'=\epsilon,
\end{aligned}
$}
\end{equation}
defined for $\tau \in [-\tau_*(\sqrt{\epsilon}),\tau_*(\sqrt{\epsilon})]$. We write $\Delta(\tau)$ for the quantity $\Delta(g(\tau),\epsilon)$, suppressing the dependency on $\epsilon$ and $g$.
\begin{proposition}\label{DiagonalCoordinates}
    Fix any $n \in \N$. For $0\ll\epsilon<1$ sufficiently small, there exists a coordinate transformation of the form $(\Delta_z,\Delta_w,\tau) \mapsto (a,b,\tau)$, defined by 
\begin{equation}  \label{Diagonalisation1}
    \begin{pmatrix}
        \Delta_z\\
        \Delta_w
    \end{pmatrix}
    =
    \begin{pmatrix}
       1 & \epsilon Y(\tau,\epsilon)\\
         \epsilon \overline Y(-\tau,\epsilon)& 1
    \end{pmatrix}
    \begin{pmatrix}
        a\\
        b
    \end{pmatrix}, \quad \tau \in [-\tau_*(\sqrt{\epsilon}),\tau_*(\sqrt{\epsilon})],
\end{equation}
with $Y$ being $C^n$-smooth, 
that brings \eqref{Delta1} into the following diagonal form 
\begin{equation} \eqlab{DiagonalSystemFinal}
    \resizebox{0.9\linewidth}{!}{$
    \begin{aligned}
    & \epsilon \frac{da}{d\tau} = g'(\tau)\left(\lambda(g(\tau))+\epsilon\partial_zF(z_b,w_b,g(\tau),\epsilon)+\epsilon^2 \partial_wF(z_b,w_b,g(\tau),\epsilon)\overline{Y}(-\tau,\epsilon)\right)a,\\
    & \epsilon \frac{db}{d\tau} = g'(\tau)\left(\nu(g(\tau))+\epsilon\partial_wG(z_b,w_b,g(\tau),\epsilon)+\epsilon^2\partial_zG(z_b,w_b,g(\tau),\epsilon)Y(\tau,\epsilon)\right)b.
\end{aligned}
    $}
\end{equation}
This system is symmetric (reversible) with respect to $(a,b,\tau) \mapsto (\overline b, \overline a, - \tau)$. 
\end{proposition}
\begin{remark}
    As demonstrated in the proof, $Y$ is in fact $C^n$-smooth with respect to $(\tau,\epsilon)$ on $I\times [0,\epsilon_0)$, with $I\subset (-1,1)$ compact, $0<\epsilon_0\ll 1$ and $n \in \N$.  Near $(\tau,\epsilon) =(\pm \tau_*,0)$, the smoothness of $Y$ should be understood with respect to the blowup coordinates associated with the following (real) blowups of $(\tau,\epsilon)=\pm 1$ given by 
    \begin{align}\eqlab{tauepsbu}
        \begin{cases}
            \tau = \pm1 + \sigma \breve \tau,\\
            \epsilon = \sigma^2 \breve \epsilon,
        \end{cases}\quad \sigma\ge 0,\,(\breve \tau,\breve\epsilon)\in \mathbb S^1\subset \mathbb R^2,
    \end{align}
    respectively.
    This is due to the fact that the smoothness of $(z^{a,r},w^{a,r})(\mu,\epsilon)$ (and therefore also $(z_b,w_b)(\mu,\epsilon)$) near $(\mu,\epsilon)=(\mu_*,0)$ and $(\mu,\epsilon)=(\overline{\mu_*},0)$ is given with respect to the blowup coordinates associated with the blowup of $(\mu,\epsilon)=(\mu_*,0)$ and $(\mu,\epsilon)=(\overline \mu_*,0)$, respectively. In line with \cite{hayes2015a}, there is a natural transition from smoothness with respect to $\epsilon$ to smoothness with respect to $\sqrt{\epsilon}$ and the blowup coordinates capture this transition, see also \cite[Proposition 5.2]{kristiansenzerophopf2}.
    
\end{remark}
\begin{proof}[Proof of Proposition 5.1]
   Consider the projective coordinate 
    \begin{equation}\nonumber
        \widetilde{\Delta} = \frac{\Delta_z}{\Delta_w},
    \end{equation}
    for $\tau\in [-1+\delta,1-\delta]$ with $\delta$ sufficiently small. Write  $H(z_b,w_b,g(\tau),\epsilon) = H$ for $H \in \{F,G\}$.
   From system \eqref{gdsys}, we obtain the following slow-fast system for $(\widetilde{\Delta},\tau)$
\begin{equation}\eqlab{projectivesystem}
    \resizebox{0.9\linewidth}{!}{$
    \begin{aligned}
        & \widetilde{\Delta}' = \epsilon g'(\tau)\partial_wF+g'(\tau)[\lambda(g(\tau))-\nu(g(\tau))+\epsilon(\partial_zF-\partial_wG)]\widetilde{\Delta}+\epsilon g'(\tau)\partial_zG\widetilde{\Delta}^2,\\
    & \tau ' = \epsilon,
    \end{aligned}
    $}
\end{equation}
for $\tau \in [-1+\delta,1-\delta]$.
By taking $\epsilon = 0$, we obtain the layer problem
\begin{equation}\eqlab{layerproblem}
    \begin{aligned}
        & \widetilde{\Delta}' = g'(\tau)[\lambda(g(\tau))-\nu(g(\tau))]\widetilde{\Delta},\\
    & \tau ' = 0.
    \end{aligned}
\end{equation}
Hence, $\widetilde{\Delta}=0$ defines a critical manifold with normal hyperbolicity described by the quantity
\begin{equation}\eqlab{Eigenvalue1}
    \real\left[g'(\tau)[\lambda(g(\tau))-\nu(g(\tau))]\right].
\end{equation}
We have the following result.
\begin{lemma}\label{Hyplemma}
    The following holds true for all $\tau \in [-1,1]$:
    \begin{equation}
        \real\left[g'(\tau)[\lambda(g(\tau))-\nu(g(\tau))]\right]>0.
    \end{equation}
\end{lemma}
(We delay the proof of Lemma \ref{Hyplemma} to the end of the present proof).

As a result of Lemma \ref{Hyplemma}, Fenichel's theory \cite{fen1,fen2,fen3,jones_1995} guarantees the existence of a repelling slow manifold
\begin{equation}
    \widetilde\Delta = \epsilon Y(\tau,\epsilon), \quad \tau \in [-1+\delta,1-\delta],
\end{equation}
for $0<\epsilon\ll 1 $ and $\delta>0$ fixed.
 Looking at order by order in the invariance equation for (\theequation), we find that
\begin{equation}\nonumber
\begin{aligned}
    Y(\tau,0)= -\frac{\partial_wF(0,0,g(\tau),0)}{\lambda(g(\tau))-\nu(g(\tau))}.
\end{aligned}
\end{equation}
Next, we wish to extend the slow manifold beyond $\tau \in [-1 + \delta, 1-\delta]$ to include $\tau = \pm \tau_*(\sqrt{\epsilon})$. By a reparameterization of $g(\tau)$, we can for $\delta>0$ small enough write:
\begin{align*}
    g(\tau ) = \mu_* + (1-\tau) \e^{\i\hat \theta(1-\tau)},
\end{align*}
(using polar coordinates centered at $\mu_*$) with $\hat \theta(1-\tau)$ a smooth function of $\tau \in [1-\delta,1]$ (with $ \hat{\theta}(0) = \theta^{\kC}$). Notice that with this choice of parametrization, we have $\tau_*(\sqrt{\epsilon}) = 1 - \sqrt{\delta^{-1}\epsilon}$. By Proposition \ref{finalprop} (using $r_1 = 1-\tau$), the intermediary points $z_b,w_b$ restricted to the contour $\kC_*$ can be written as:
\begin{equation}\eqlab{intersmooth}
\begin{aligned}
    z_b(g(\tau),\epsilon)& = (1-\tau)  z_{b1}(\hat{\theta}(1-\tau),1-\tau,\epsilon_1), \\ w_b(g(\tau),\epsilon) &= (1-\tau)w_{b1}(\hat{\theta}(1-\tau),1-\tau,\epsilon_1),
    \end{aligned}
\end{equation}
where  $\epsilon_1 = (1-\tau)^{-2} \epsilon$.
Here the right hand side is smooth with respect to $(\tau,\epsilon_1)\in [1-\delta,1]\times [0,\delta]$.
We now write \eqref{projectivesystem} (with $\epsilon'=0$ augmented) with respect to the coordinates $(\sigma_1,\epsilon_1)$ defined by the directional blowup (recall Remark 5.2):
\begin{equation}\eqlab{EntryExitB}
    \begin{aligned}
        &\tau = 1 - \sigma_1,\\
        &\epsilon = \sigma_1^2 \epsilon_1.
    \end{aligned}
\end{equation}
This leads to the following system:
\begin{equation}\eqlab{dgdsys}
    \begin{aligned}
     \widetilde \Delta' =& g'(1-\sigma_1)\left[\lambda(g(1-\sigma_1))-\nu(g(1-\sigma_1))+{\sigma_1^2\epsilon_1}(\partial_zF-\partial_wG)\right] \widetilde \Delta\\
    & +g'(1-\sigma_1){\sigma_1^2\epsilon_1}\left((\partial_wF)-(\partial_zG) \widetilde \Delta^2\right),\\
    \sigma_1' =& - \sigma_1^2\epsilon_1,\\
    \epsilon_1' =& 2\sigma_1\epsilon_1^2.
    \end{aligned}
\end{equation}
with the right hand side being smooth by \eqref{intersmooth} with respect to $(\sigma_1,\epsilon_1)\in [0,\delta]\times[0,\delta]$. For $\epsilon_1=\sigma_1=0$, we have:
\begin{equation}
    \widetilde  \Delta' = g'(1)(\lambda(\mu_*)-\nu(\mu_*))\widetilde \Delta = -g'(1)\nu(\mu_*)\widetilde \Delta.
\end{equation}
By Lemma \ref{Hyplemma}, we have a semi hyperbolic point at $(\sigma_1,\epsilon_1,\widetilde \Delta) = (0,0,0)$ and a two-dimensional repelling center manifold:
\begin{equation}\eqlab{centermanifold1}
    \widetilde \Delta = \sigma_1^2\epsilon_1 Y_1(\sigma_1,\epsilon_1), \quad Y_1(\sigma_1,0) = Y(1-\sigma_1,0),
\end{equation}
with $Y_1$ being $C^n$-smooth on $(\sigma_1,\epsilon_1) \in [0,\delta]\times[0,\delta]$, $\delta>0$ sufficiently small.
The extension near the point $\tau = -1$ is obtained by considering the directional chart
\begin{equation}\eqlab{directional2}
    \begin{aligned}
        &\tau = -1 + \sigma_2,\\
        &\epsilon = \sigma_2^2\epsilon_2.
    \end{aligned}
\end{equation}
Here we also find a $C^n$-smooth repelling center manifold  as a graph over $(\sigma_2,\epsilon_2) \in [0,\delta]\times[0,\delta]$.
The details are similar and therefore left out.

Notice that the center manifolds in the entry/exit charts as well as the slow manifolds on the compact interval are all attracting with respect to the backward flow. Moreover, under the backward flow, $\tau$ is also decreasing since $\tau' = - \epsilon$ (taking the derivative with respect to $-t$).

We therefore fix a center manifold \eqref{centermanifold1} in the $(\sigma_1,\epsilon_1)$-coordinates and then extend this as a slow manifold of \eqref{projectivesystem} (with $\epsilon'=0$ augmented) to $\tau \in [-1+\delta,1-\delta]$ by using the backward flow and the smooth change of coordinates. Finally, the invariant manifold is extended into the coordinates \eqref{directional2} by using the associated center manifold.
As a result, we obtain a global invariant manifold $$ \widetilde \Delta = \epsilon Y(\tau,\epsilon), \quad \tau \in [-\tau_*(\sqrt{\epsilon}),\tau_*(\sqrt{\epsilon})],$$ with $Y(\tau,\epsilon)$ is $C^n$-smooth on $[-1+\delta,1-\delta]\times [0,\epsilon_0)$ (as a slow manifold). Close to $\tau_*$, we have that
\begin{equation}\eqlab{johantræt2}
    Y(\tau,\epsilon) =Y_1(1-\tau,\frac{\epsilon}{(1-\tau)^2}), \quad \tau \in [1-\delta,1-\sqrt{(\delta^{-1}\epsilon)}],
\end{equation}
with $Y_1$ being $C^n$-smooth. We have a similar expansion near $\tau=-\tau_*(\sqrt{\epsilon})$.

 
We conclude that the system \eqref{Delta1} has an invariant manifold of the form $\Delta_z = \epsilon Y(\tau,\epsilon) \Delta_w$. Now, using the symmetry $(z,w,g(\tau)) \rightarrow (\overline w, \overline z, g(-\tau))$, see \eqref{CSYM}, we find a separate invariant manifold of \eqref{Delta1} 
\begin{equation}\nonumber
    \Delta_w = \epsilon\overline{Y}(-\tau,\epsilon)\Delta_z.
\end{equation}
It then follows, that the change of coordinates
\begin{equation}  \label{DiagonalisationProof}
    \begin{pmatrix}
        \Delta_z\\
        \Delta_w
    \end{pmatrix}
    =
    \begin{pmatrix}
       1 & \epsilon Y(\tau,\epsilon)\\
         \epsilon \overline{Y}(-\tau,\epsilon)& 1
    \end{pmatrix}
    \begin{pmatrix}
        a\\
        b
    \end{pmatrix}, \quad \tau \in [-\tau_*(\sqrt{\epsilon}),\tau_*(\sqrt{\epsilon})],
\end{equation}
diagonalizes \eqref{projectivesystem} by construction. Indeed, the columns of (\ref{DiagonalisationProof}) are invariant manifold solutions. The $(a,b,\tau)$ system in \eqref{DiagonalSystemFinal} is then obtained by considering $a=0$ ($b=0$, respectively):
\begin{align*}
    \Delta_z = a, \quad \Delta_w = \epsilon \overline{Y}(-\tau,\epsilon)a.
\end{align*}
Taking the derivative with respect to $\tau$, we find
\begin{align*}
    \epsilon \frac{da}{d\tau} = g'(\tau)\left(\lambda(g(\tau))+\epsilon\partial_zF(z_b,w_b,g(\tau),\epsilon)+\epsilon^2 \partial_wF(z_b,w_b,g(\tau),\epsilon)\overline{Y}(-\tau,\epsilon)\right)a.
\end{align*}
We have a similar result for $b$, which finalizes the proof.

\end{proof}

\begin{proof}[Proof of Lemma \ref{Hyplemma}]
    Due to \eqref{CSYM} and \eqref{symFL}, we see that the expression
    \begin{equation}\eqlab{Eigenvalue3}
    \real\left[g'(\tau)[\lambda(g(\tau))-\nu(g(\tau))]\right].
\end{equation}
    is an even function of $\tau$. Hence it suffices to consider $\tau\in [0,1]$.
We now show that \eqref{Eigenvalue3} is non-zero (positive) on $[0,1]$. Since $\lambda(\mu)$ only vanishes at $\mu_*$, we can write
\begin{equation}\eqlab{Tangent}
    g'(\tau) = c_1(\tau) \frac{\overline{ \lambda(g(\tau))}}{\vert \lambda(g(\tau)) \vert }+c_2(\tau)\i \frac{\overline{\lambda(g(\tau))}}{\vert \lambda(g(\tau)) \vert},\quad \tau\in [0,1).
\end{equation}
In particular, as the root at $\mu_*$ is simple, the real-valued functions $c_1$ and $c_2$ can be $C^\infty$-smoothly extended to $\tau\in [0,1]$. Moreover, since there are no contact points (recall assumption \ref{B2}), $c_1(\tau) \neq 0$ for all $\tau \in [0,1]$. Finally, as $\lambda(0)$ and $g'(0)$ are both purely imaginary (the former being due to the Hopf while the latter follows from the fact that $\overline{\kC_*}=\kC_*$), we conclude that $c_1(\tau)>0$ for all $\tau \in [0,1]$ (and that $c_2(0)=0$). 
 
Using the expression of $g'(r)$ given in \eqref{Tangent}, suppressing for simplicity the explicit dependence on $g(\tau)$ on the right hand side, \eqref{Eigenvalue3} becomes
\begin{align}
    &\real\left[g'(\tau)[\lambda(g(\tau))-\nu(g(\tau))]\right] = c_1 \vert \lambda \vert^{-1}(\vert \lambda \vert^2 - \real(\overline \lambda \nu)). \eqlab{Eigenvalue2}
\end{align}
Now, recall that for points $g(\tau) \in \kC_*\subset \kC$, we have that 
\begin{equation}\eqlab{omg}
    \real [ -\i \overline{\lambda(g(\tau))}\nu(g(\tau))] = 0.
\end{equation}
At the same time, $\nu(\mu)$ is non-zero for $\imag(\mu) \ge 0$ while $\lambda(\mu)=0$ only at $\mu = \mu_*$, recall assumption \ref{A1} and \eqref{symFL}, further implying that
\begin{equation}\eqlab{Prop1}
    \lambda(g(\tau)) = c_3(\tau) \nu(g(\tau)),
\end{equation}
with $c_3(\tau)$ a $C^\infty$ real-valued function, satisfying $c_3(1)=0$. Indeed, we can write $\lambda(g(\tau)) = c_3(\tau) \nu(g(\tau))+ic_4(\tau) \nu(g(\tau))$, which inserted into \eqref{omg} gives $c_4=0$. Similarly, at the origin $\tau =0$, we find
\begin{align}
    c_3(0) = -1,
\end{align}
using that $\lambda(0) = - \nu(0) \neq 0$.
We are led to conclude that the only root of $c_3(\tau)$ is at $\tau = 1$. In particular, $c_3(\tau)<0$ for $\tau \in [0,1)$. Finally, we combine \eqref{Prop1} with \eqref{Eigenvalue2}:
\begin{equation}\eqlab{Result1}
    \real\left[g'(\tau)[\lambda(g(\tau))-\nu(g(\tau))]\right] = c_1\vert c_3 \vert^{-1} \vert \nu \vert^{-1} (\vert c_3 \vert^2 \vert \nu \vert^2 - \real( c_3 \vert \nu \vert ^2)) = c_1\vert \nu \vert(1-c_3),
\end{equation}
using that $c_3(\tau) \leq 0$.
We have that  $\vert \nu(g(\tau))\vert > 0$, $c_1(\tau) > 0$, and $c_3(\tau) \leq 0$ for all $\tau \in [0,1]$, leading to the conclusion that the quantity \eqref{Eigenvalue1} is strictly positive for $\tau \in [-1,1]$.
\end{proof}

We can now finalize the proof of Theorem \ref{TheoremSplitting} 
\begin{proof}[Proof of Theorem \ref{TheoremSplitting}]
    Using the coordinates of Proposition \ref{DiagonalCoordinates}, the system takes the form
\begin{align*}
    &\frac{da}{d\tau} =g'(\tau) \left(\frac{ \lambda(g(\tau))}{\epsilon}+\partial_zF(0,0,g(\tau),0)+\kO(\sqrt{\epsilon})\right)a,\\
    &\frac{db}{d\tau} = g'(\tau)\left(\frac{ \nu(g(\tau))}{\epsilon}+\partial_wG(0,0,g(\tau),0)+\kO(\sqrt{\epsilon})\right)b.
\end{align*}
Here, we have expanded $(z_b,w_b)$ and $Y(\tau,\epsilon)$ with respect to $\sqrt{\epsilon}$ using the smoothness properties shown in the proof of Proposition 5.1.
Using the form of $g'(\tau)$ in \eqref{Tangent}, we find
\begin{equation}\eqlab{Integrands}
    \begin{aligned}
        &g'(\tau)\lambda(g(\tau)) = (c_1(\tau)+\i c_2(\tau))\vert \lambda (g(\tau)) \vert,\\
        &g'(\tau)\nu(g(\tau)) = \vert \nu(g(\tau)) \vert \frac{c_3(\tau)}{\vert c_3(\tau) \vert}(c_1(\tau)+i c_2(\tau))=-\vert \nu(g(\tau))\vert (c_1(\tau)+i c_2(\tau)).
    \end{aligned}
\end{equation}
We immediately see that the first equation of \eqref{Integrands} has positive real part, since $c_1(\tau)$ is positive on $\tau \in [0,1]$. Similarly, the second equation has negative real part on $\tau \in [0,1)$, due to the fact that $c_3(\tau) < 0$ for $\tau \in [0,1)$. Hence
\begin{align*}
    \real \left[\int_0^1 g'(\tau) \lambda(g(\tau)) d\tau \right]=\real \left[\int_0^{\mu_*} \lambda(\mu) d\mu \right]>0,
\end{align*}
upon using Cauchy's theorem. This proves \eqref{negativequantity}. 
We therefore focus on the expression for $a(\tau)$, $\tau \in [0,\tau_*]$. We solve the scalar differential equation for $a(\tau)$ via an integrating factor 
\begin{align}
        a(0) &= a(\tau_*)\exp\left({-\int_0^{\tau_*} \left[\frac{g'(\tau)\lambda(g(\tau))}{\epsilon}+g'(\tau)\partial_zF(0,0,g(\tau),0)+\kO(\sqrt{\epsilon})\right]d\tau}\nonumber \right)\\
    & = a(\tau_*)\exp\left({-\int_0^{\mu_*+\sqrt{\epsilon}\hat \mu} \left[\frac{\lambda(\mu)}{\epsilon}+\partial_zF(0,0,\mu,0)\right]d\mu}\right)(1+\kO(\sqrt{\epsilon})),\eqlab{a0}
\end{align}
using Cauchy's theorem in the second equality. Now, using \eqref{hatlambda}  and the smoothness of $\hat \mu(\sqrt{\epsilon})$, we obtain 
\begin{equation*}
\begin{aligned}
    \int_0^{\mu_*+\sqrt{\epsilon}\hat \mu}\lambda(\mu)d\mu & = \int_0^{\mu_*} \lambda(\mu) d \mu + \frac{\lambda_0}{2\delta}\e^{2\i \theta^{\kC}} \epsilon + \kO(\epsilon^{\frac{3}{2}}),
\end{aligned}
\end{equation*}
by expanding in powers of $\sqrt{\epsilon}$, recall also \eqref{mufinal}. Then a similar argument yields:
\begin{equation*}
    \int_0^{\mu_*+\sqrt{\epsilon}\hat \mu} \partial_z F(0,0,\mu,0) d\mu = \int_0^{\mu_*} \partial_z F(0,0,\mu,0) d\mu + \mathcal O(\sqrt{\epsilon}).
\end{equation*}
From the result of Proposition \ref{DiagonalCoordinates}, we know that 
\begin{align*}
    &a(\tau) = \frac{1}{1-\epsilon^2Y(\tau,\epsilon)\overline{Y}(-\tau,\epsilon)}(\Delta_z(\tau)-\epsilon Y(\tau,\epsilon)\Delta_w(\tau)) = \Delta_z(\tau) + \kO(\epsilon),
\end{align*}
in particular for $\tau = \tau_*$ 
\begin{equation}\nonumber
    a(\tau_*) = \sqrt{2\pi\epsilon}(-1)^{\ell-1} \frac{F_0}{\sqrt{\vert \lambda_0 \vert}}\e^{-\frac{3\i}{2} \arg(\lambda_0)} \e^{\frac{\lambda_0}{2\delta}\e^{2\i\theta^{\kC}}}(1+\kO(\sqrt{\epsilon})),
\end{equation}
using the expression \eqref{Deltatau} for $\Delta_z(\tau_*)$ and the fact that $F_0\ne 0$. This brings \eqref{a0} into the following form:
\begin{equation*}
\resizebox{0.99\linewidth}{!}{$
    a(0) =  \sqrt{(2\pi \epsilon)}\frac{F_0}{\sqrt{\vert\lambda_0\vert}}(-1)^{\ell -1} \e^{- \frac{3\i }{2}\arg(\lambda_0)}  \exp\left({-\int_0^{\mu_*} \left[\frac{\lambda(\mu)}{\epsilon}+\partial_zF(0,0,\mu,0)\right]d\mu}\right) (1+\kO(\sqrt{\epsilon})). 
    $}
\end{equation*}
We note that the leading order expression is independent of $\delta$ (as desired).
By the symmetry $(a,b,\tau) \mapsto (\overline b, \overline a, - \tau)$ of the \eqref{DiagonalSystemFinal}, we obtain the following 
\begin{align}
    b(0) &= \overline{a(0)}.
\end{align}
Finally, transforming back to the $\Delta$-coordinates yields the formula 
\begin{align}
    \Delta_z(0) =& a(0) + \epsilon Y(0,\epsilon) \overline{a(0)} \nonumber\\
    =&\sqrt{(2\pi \epsilon)}\frac{F_0}{\sqrt{\vert\lambda_0\vert}}(-1)^{\ell-1}\e^{-\frac{3\i}{2}\arg(\lambda_0)} \times \\ &\exp\left({-\int_0^{\mu_*} \left[\frac{\lambda(\mu)}{\epsilon}+\partial_zF(0,0,\mu,0)\right]d\mu}\right)(1+\kO(\sqrt{\epsilon})), \eqlab{SIC}
\end{align}
and $\Delta_w(0) = \overline{\Delta_z(0)}$. 
To complete the proof, we use that $\Delta_x = \real(\Delta_z)$ and $\Delta_y = \imag(\Delta_z)$.
\end{proof}


\section{Proof of Proposition \ref{mainprop}}\label{proofprop}
In this section, we prove Proposition \ref{mainprop}. We therefore consider \eqref{ComplexI} with $\lambda$ and $\nu$ being affine functions of $\mu$, see \eqref{lambdanulin} repeated here for convenience
\begin{align}\eqlab{lambdanulin2}
    \begin{cases} \lambda(\mu) = \lambda_0\mu -\i,\\
     \nu(\mu) = \overline{\lambda_0}\mu +\i, 
     \end{cases}
\end{align}
with $\real(\lambda_0)>0.$  In this case, the associated elliptic system \eqref{THIS} takes the following form
\begin{equation}\eqlab{muElLin}
    \dot \mu = -\i\overline{\lambda(\mu)} =-\i(\overline{\lambda_0} \overline{\mu}+\i),
\end{equation}
with
\begin{align}
    \mu_*=\frac{\i}{\lambda_0}.\eqlab{mustarlin}
\end{align}
To prove Proposition \ref{mainprop}, we consider the assumptions successively. Firstly,
    \ref{A1} is trivial.
    For \ref{A2} and \ref{A3}, we use the conservation of $\mathcal H(\mu)$, see \eqref{Hmu}, to obtain the following expression for the invariant manifolds of $\mu_*$, recall Lemma \ref{Lemma21}, as level curves 
    \begin{align*}
        0 &= \real\left[\int_{\mu^*}^\mu \lambda(s)ds\right] = \real \left( \frac12 \lambda_0 (\mu^2-\mu_*^2)-i(\mu-\mu_*)\right).
    \end{align*}
     For $\mu\in \mathbb R$, a simple calculation, using \eqref{mustarlin} and $\real(\lambda_0)>0$, shows that this equation reduces to 
    \begin{align*}
        \mu^2 \vert \lambda_0\vert^2 =1.
    \end{align*}
    Hence 
    \begin{align*}
    \mu_s = -\mu_u=-\vert \lambda_0\vert^{-1}<0.
    \end{align*}
    Next, for \ref{B1} we compute 
    \begin{align}
       \real \left[-\i \overline{\lambda(\mu)}\nu(\mu)\right] =& \real\left[-\i(\overline{\lambda_0}^2\vert\mu\vert^2+\i\overline{\lambda_0}(\mu+\overline{\mu})-1)\right],\nonumber\\ 
       =& -2\real(\lambda_0)\imag(\lambda_0)\vert\mu\vert^2+2\real(\lambda_0)\real(\mu)=0.\eqlab{kCeqn}
    \end{align}
    For $\imag(\lambda_0)\ne 0$, we conclude that 
    \begin{align*}
\kC = \left\{\mu\in \mathbb C \,:\, \real \left[-\i \overline{\lambda(\mu)}\nu(\mu)\right] =0\right\},
    \end{align*}
    is a circle centered at $$\mu=\mu_c:=1/(2\imag(\lambda_0)),$$ and radius $\vert \mu_c\vert>0$. For $\imag(\lambda_0)=0$, the circle degenerates to the imaginary axis. In this special case, it is easy to see that \ref{B1}-\ref{B3} are all satisfied (this case corresponds to \cite{hayes2015a}). Suppose that $\imag(\lambda_0)>0$. Then $\mu \to -\overline{\mu}$ and $t\mapsto -t$ transform \eqref{muElLin} into the same form with $\imag(\lambda_0)$ replaced by $-\imag(\lambda_0)<0$. This follows from a simple calculation. We therefore restrict attention to $\imag(\lambda_0)<0$ in the following. In this case, $\kC$ is contained within the left-half plane.
    Now, 
    evaluating $\real\left[-\i\overline{ \lambda(\mu)}\nu(\mu)\right]$ at $\mu=\mu_c$, we find by \eqref{kCeqn} that
    \begin{equation}\nonumber
        -\frac{\real(\lambda_0)}{2\imag(\lambda_0)}+\frac{\real(\lambda_0)}{\imag(\lambda_0)}=\frac{\real(\lambda_0)}{2\imag(\lambda_0)}<0.
    \end{equation}
    We conclude that $\operatorname{Re}\left[-i\overline{ \lambda(\mu)}\nu(\mu)\right]<0$ inside the circle. Finally, by \eqref{kCeqn} the map $$(a,b)\mapsto \real\left[-\i\overline{ \lambda(a+ib)}\nu(a+ib)\right],$$ is clearly regular. The closed arc from $\mu_*$ to $\overline{\mu_*}$ through $0$ is our $\kC_*$ and hence \ref{B3} holds true for all $\lambda_0$.

    We now turn to assumption \ref{B2}. Contact points on $\kC$ are given by zeros of the Lie-derivative of the left hand side of \eqref{kCeqn} with respect to $\mu$ in the direction of \eqref{muElLin}. This gives three contact points: $\mu=\mu_*$, $\mu=\mu_{ct}$ and $\mu=\overline{\mu_{ct}}$ where
    \begin{align*}
        \mu_{ct} := \frac{1}{4\imag(\lambda_0)}-i \frac{\sqrt{3}}{4\imag(\lambda_0)}.
    \end{align*}
    This follows from an elementary calculation.
    Recall here that we restrict attention to $\imag(\lambda_0)<0$.
We see that $\mu_{ct}=\mu_*$ (it suffices to look at the real part since both points lie on $\kC$) if and only if 
\begin{align}
        \real(\lambda_0) = -\sqrt{3}\imag(\lambda_0).\eqlab{cond0}
    \end{align}
Notice that this condition is equivalent with $\mu_s=\mu_c$ since (a) the radius from $\mu_c$ to $\mu_*$ defines a normal to $\kC$ at $\mu_*$ and (b) the stable and unstable manifolds are perpendicular at $\mu_*$. Hence $\kC_*$ is tangent to the unstable manifold at $\mu_*$ if and only if \eqref{cond0} holds. For $0<\real(\lambda_0)< -\sqrt{3}\imag(\lambda_0)$, we have $\real(\mu_*)<\real(\mu_{ct})$ as illustrated in Figure \ref{fig:second}. In this situation, the contact point  $\mu_{ct}$ (purple point in Figure \ref{fig:second}) is therefore on the interior of $\kC_*$. (In fact, it follows directly from Rolle's theorem that $\mu_{ct}$ lies on the arc between $\mu_*$ and the intersection point with the unstable manifold.) We conclude that \ref{B2} is violated for  any $0<\real(\lambda_0)\le  -\sqrt{3}\imag(\lambda_0)$. For $\real(\lambda_0)> -\sqrt{3}\imag(\lambda_0)$, we have $\real(\mu_*)>\real(\mu_{ct})$ as illustrated in Figure \ref{fig:first}. Hence there are no contact points on $\kC_*$ in this case and \ref{B2} is satisfied. This completes the proof. 
\begin{figure}
    \centering
      \begin{subfigure}[b]{0.45\textwidth}
        \centering
        \includegraphics[width=\linewidth]{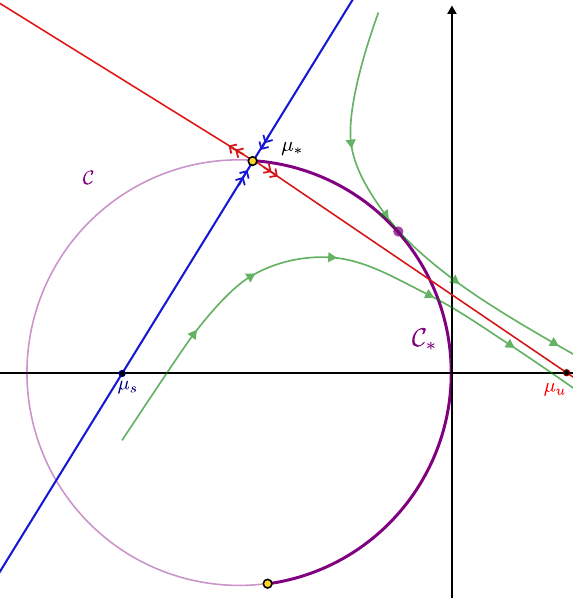}
        \caption{$0<{\real(\lambda_0)}<-\sqrt{3}{\imag(\lambda_0)}$}
       \label{fig:second}
    \end{subfigure}
        \hfill\begin{subfigure}[b]{0.475\textwidth}
        \centering
        \includegraphics[width=\linewidth]{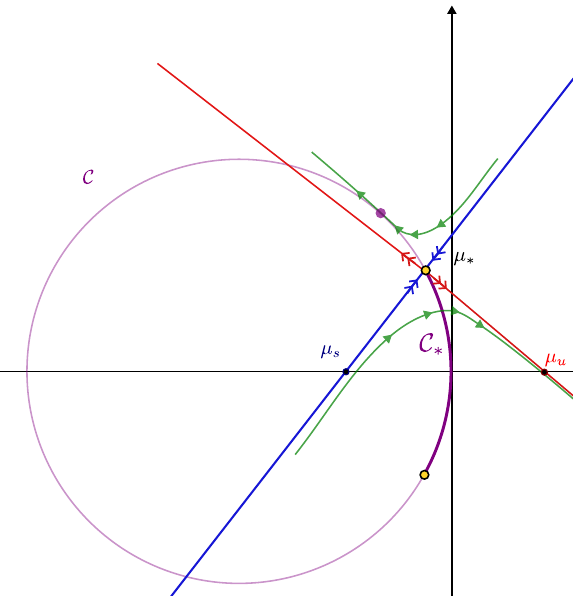}
        \caption{${\real(\lambda_0)}>-\sqrt{3}{\imag(\lambda_0)}$}
        \label{fig:first}
    \end{subfigure}
    \caption{The elliptic system and the contour $\kC$ for the linear system \eqref{muElLin} for $\imag (\lambda_0)<0$ and $\real( \lambda_0)>0$. In case (a), there is a contact point (purple disc) along $\kC_*$, whereas in case (b) there are no contact points along $\kC_*$. (Notice that $\mu_u=-\mu_s$ for \eqref{muElLin} but for illustrative purposes we have chosen to  ignore this in the present figure.) }
    \label{KSystemFig}
\end{figure}

\section{Proof of Corollary \ref{Delay}}\label{secdelay}
   In this section, we finally turn to the proof of Corollary \ref{Delay} and the expression for the maximal delay, see \eqref{Deq0}. For this, we consider the real system \eqref{ReS1} and write 
    \begin{align*}
        \Delta(\mu,\epsilon)=(\Delta_x,\Delta_y)(\mu,\epsilon) = \xi^a(\mu,\epsilon)-\xi^r(\mu,\epsilon),
    \end{align*}
    with $\xi^{a,r}=(x^{a,r},y^{a,r})$. Notice that $\xi^r(\mu,\epsilon)$ is well-defined for $\mu\in [0,\mu_u+\delta]$. Moreover, $\Delta(0,\epsilon)$ is known from Theorem \ref{TheoremSplitting}. A simple calculation, shows that $\Delta$ satisfies the following differential system 
    \begin{align}\nonumber
    & \epsilon\frac{d\Delta}{d\mu} = 
    \begin{pmatrix}
        \alpha(\mu)&\beta(\mu)\\
        -\beta(\mu) & \alpha(\mu)
    \end{pmatrix}
    \Delta + \epsilon
    \begin{pmatrix}
        R(x^r+\Delta_x,y^r+\Delta_y,\mu,\epsilon)-R(x^r,y^r,\mu,\epsilon)\\
        S(x^r+\Delta_x,y^r+\Delta_y,\mu,\epsilon)-S(x^r,y^r,\mu,\epsilon),
    \end{pmatrix}.
\end{align}
We write this equation in the following form
\begin{align}
    & \epsilon\frac{d\Delta}{d\mu} = 
    \left(\begin{pmatrix}
        \alpha(\mu)&\beta(\mu)\\
        -\beta(\mu) & \alpha(\mu)
    \end{pmatrix}
    + \epsilon
    T(\Delta,\mu,\epsilon) \right)\Delta,\eqlab{comeon}
\end{align}
by using the mean value theorem. Clearly, 
 $\vert\Delta(\mu,\epsilon)\vert$ is monotonically increasing for $\mu>0$ uniformly bounded away from $0$, since we are on the repelling side of the Hopf bifurcation (where $\alpha(\mu)>0$, recall assumption \ref{A1}) for all $0<\epsilon\ll 1$. The existence and uniqueness of $\mu_\delta(\epsilon)>0$ such that 
  \begin{align*}
      \Delta(\mu_\delta(\epsilon),\epsilon)=\delta,
  \end{align*}
  with $\delta>0$ fixed small enough and $0<\epsilon\ll 1$, therefore follows. In the following, we will prove the asymptotic expansion of $\mu_\delta(\epsilon)$ in \eqref{Deq0}. We henceforth write $\Delta(\mu,\epsilon)$ as $\Delta(\mu)$.

Now, define $\widetilde \Delta(\mu)$ by 
\begin{equation}\eqlab{TD1}
     \Delta(\mu) =\sqrt{\epsilon} \exp\left({\epsilon}^{-1}\left(\int_0^{\mu} \alpha(s)ds-\real\left[\int_0^{\mu_*} \lambda(s)ds\right] \right)\right)\Omega(\mu,\epsilon)
    \widetilde \Delta(\mu),
\end{equation}
with 
\begin{align*}
    \Omega(\mu,\epsilon) = \exp\left({\epsilon}^{-1} \int_0^{\mu} \begin{pmatrix}
        0&\beta(s)\\
        -\beta(s) & 0
    \end{pmatrix}ds\right).
\end{align*}
Notice from Theorem \ref{TheoremSplitting} that $\widetilde \Delta(0)=\mathcal O(1)$ and that $\Omega(\mu,\epsilon)\in O(2)$ is a rotation matrix for any $\mu\in \Sigma \cap \mathbb R$, $\epsilon>0$.
The equation \eqref{TD1} defines a change of coordinates that directly brings \eqref{comeon} into the following form  
\begin{equation}\eqlab{Ddeq}
    \widetilde \Delta'(\mu) =\Omega(\mu,\epsilon)^{-1} T(\Delta,\mu,\epsilon)\Omega(\mu,\epsilon) \widetilde \Delta(\mu).
\end{equation}
From here a straightforward application of Gronwall's inequality leads to a uniform estimate: $\widetilde \Delta(\mu)=\kO(1)$ (while $\vert \Delta(\mu)\vert$ remains bounded). From \eqref{Ddeq}, we then in turn also have that $\widetilde \Delta'(\mu)=\kO(1)$ under the same conditions. Here we have used that $\Omega(\mu,\epsilon)$ is a rotation matrix with $\norm{\Omega}=1$.

Now, given \eqref{TD1}, we can write $\vert \Delta(\mu) \vert = \delta$
as follows 
\begin{align*}
  \sqrt{\epsilon} \exp\left({\epsilon}^{-1}\left(\int_0^{\mu} \alpha(s)ds-\real\left[\int_0^{\mu_*} \lambda(s)ds\right] \right)\right)\vert \widetilde \Delta(\mu)\vert = \delta.
\end{align*}
%
Taking $\log$ on both sides and rearranging (using Cauchy's theorem and multiplication by $\epsilon$) lead to 
\begin{equation}\eqlab{Cd1}
    \mathcal H(\mu) + \epsilon\log(\sqrt{\epsilon})+ \epsilon \log \vert \widetilde \Delta(\mu)\vert-\epsilon \log \delta=0,
\end{equation}
where $$\mathcal H(\mu) = \real\left[\int_{\mu_*}^\mu \lambda(s) ds\right],$$
denotes the conserved quantity of the elliptic flow, 
recall \eqref{Hmu} and Lemma \ref{lemmaHam}. Since $\mathcal H(\mu_*)=0$, we have $\mathcal H(\mu)=0$ on the unstable manifold of the elliptic system. Recall that $\mu_u$ denotes the intersection of the unstable manifold of $\mu_*$ for the elliptic system with real axis, see assumptions \ref{A1}-\ref{A3}. Hence $\mu=\mu_u$ is a simple root of $\mathcal H$. The fact that the root is simple follows from 
\begin{align}\eqlab{Hprime}
\mathcal H'(\mu_u)=\alpha(\mu_u)>0,
\end{align}
using assumption \ref{A1}. Hence, we obtain 
\begin{align}\eqlab{mudelta0}\mu_\delta(\epsilon)=\mu_u+o(1),
\end{align} by the implicit function theorem. (This also follows from \cite{neishtadt1987a}.)

Next, in order to determine the leading order of the $o(1)$-term in \eqref{mudelta0}, we introduce $\widetilde \mu$ through the scaling 
\begin{align*}
    \mu = \mu_u + \left(\epsilon \log \sqrt{\epsilon}\right)\widetilde \mu.
\end{align*}
Inserting this into \eqref{Cd1} and using $$\mathcal H(\mu) = \alpha(\mu_u)(\mu-\mu_u)+\mathcal O((\mu-\mu_u)^2),$$ recall \eqref{Hprime}, yields
\begin{align}\eqlab{finaltildemu}
    \widetilde \mu \alpha( \mu_u)+ 1=o(1),
\end{align}
after some simple manipulations. 
Here $o(1)$ denotes a $C^1$-function with respect to $\widetilde \mu$ that is zero for $\epsilon\to 0$. Setting $\epsilon=0$ in \eqref{finaltildemu}, we therefore obtain a simple root given by 
  %
  %
\begin{equation}\nonumber
    \widetilde \mu = -\frac{1}{\alpha(\mu_u)}.
\end{equation}
Then by applying the implicit function theorem and reverting back to $\mu$, we complete the proof of Corollary \ref{Delay}.

\section*{Acknowledgments}
The authors are grateful for fruitful discussions with Martin Wechselberger. 
The authors declare that AI has not been used in any form and they assume full responsibility for all content.

\newpage
\appendix
\section{A local normal form}\applab{app}
In this section, we show that \eqref{ReS1} is a local normal form for a slow passage through a Hopf in $\mathbb R^3$. We therefore consider a general $(1,2)$ slow-fast system of the form
\begin{equation}\eqlab{SlowFastStandard}
    \begin{aligned}
        x' &= f(x,y,\mu,\epsilon),\\
        y' &= g(x,y,\mu,\epsilon),\\
        \mu' & = \epsilon h(x,y,\mu,\epsilon).
    \end{aligned}
\end{equation}
Here $f,g,h$ are assumed to be locally defined real-analytic functions with $x,y,\mu \in \R$ and $0\le \epsilon\ll 1$. We suppose that the layer problem undergoes a Hopf bifurcation for $\mu=0$. In other words, we suppose that there is a critical manifold $(x,y)=(X_0(\mu),Y_0(\mu))$, with $\mu$ in a neighborhood of $0$, where the following eigenvalues $$\alpha(\mu)\pm i\beta(\mu),$$ of the linearization 
\begin{align*}
 \begin{pmatrix}
  \partial_xf(X_0(\mu),Y_0(\mu),\mu,0) &  \partial_yf(X_0(\mu),Y_0(\mu),\mu,0)\\
  \partial_xg(X_0(\mu),Y_0(\mu),\mu,0) &  \partial_yg(X_0(\mu),Y_0(\mu),\mu,0)
 \end{pmatrix},
\end{align*}
satisfy: $$\alpha(0)=0, \quad \alpha'(0)\ne 0, \quad \beta(0)\ne 0.$$ We also assume that $$h(X_0(0),Y_0(0),0,0)\ne 0,$$ so that we have a slow passage through a Hopf bifurcation. Then, in a neighborhood of $(x,y,\mu,\epsilon)=(X_0(0),Y_0(0),0,0)$, we can divide the right hand side by $h(x,y,\mu,\epsilon)$ 
\begin{equation}\eqlab{xymu}
\begin{aligned}
 x' &= f(x,y,\mu,\epsilon),\\
 y' &= g(x,y,\mu,\epsilon),\\
 \mu' & = \epsilon,
\end{aligned}
\end{equation}
upon redefining $f$ and $g$. 
\begin{lemma}\label{SlowFastNormal}
    There is an $\epsilon$-dependent locally defined real-analytic change of coordinates $(x,y,\mu)\mapsto (\widetilde x,\widetilde y,\mu)$ ($\mu$-fibered) that brings \eqref{xymu} into the following form 
\begin{equation}\eqlab{normalform1}
\begin{aligned}
 \widetilde x' &= \alpha(\mu)\widetilde x + \beta(\mu)\widetilde y+\epsilon \widetilde R(\widetilde x,\widetilde y,\mu,\epsilon),\\
 \widetilde y' &= -\beta(\mu)\widetilde x + \alpha(\mu)\widetilde y+\epsilon \widetilde S(\widetilde x,\widetilde y,\mu,\epsilon),\\
 \mu' &=\epsilon,
\end{aligned}
\end{equation}
where $\alpha$, $\beta$, $\widetilde R$, and $\widetilde S$ are each locally defined real-analytic functions.
\end{lemma}
\begin{proof}
 We follow \cite{neishtadt1987a,neishtadt1988a} and achieve the desired change of coordinates in four steps. First, we define $(\widetilde x,\widetilde y)$ by
 \begin{align*}
  x &= X_0(\mu)+\widetilde x,\\
  y &= Y_0(\mu) + \widetilde y.
 \end{align*}
 This brings the critical manifold to the $\mu$-axis. 
We write
\begin{align*}
\resizebox{0.99\linewidth}{!}{$
 f(x,y,\mu,\epsilon) = \partial_xf(X_0,Y_0,\mu,0) \widetilde x+ \partial_yf(X_0,Y_0,\mu,0) \widetilde y+\epsilon \partial_\epsilon f(X_0,Y_0,\mu,0)+\widetilde R(\widetilde x,\widetilde y,\mu,\epsilon),$}
\end{align*}
where $\widetilde R(\widetilde x,\widetilde y,\mu,\epsilon)=\mathcal O(\vert (\widetilde x,\widetilde y,\epsilon)\vert^2)$ since $f(X_0(\mu),Y_0(\mu),\mu,0)=0$ for all $\mu\in I$. The function $g$ is expanded in a similar way 
\begin{align*}
\resizebox{0.99\linewidth}{!}{$
 g(x,y,\mu,\epsilon) = \partial_x g(X_0,Y_0,\mu,0) \widetilde x+ \partial_y g(X_0,Y_0,\mu,0) \widetilde y+\epsilon\partial_\epsilon g(X_0,Y_0,\mu,0)+\widetilde S(\widetilde x,\widetilde y,\mu,\epsilon),$}
\end{align*}
where $\widetilde S(\widetilde x,\widetilde y,\mu,\epsilon)=\mathcal O(\vert (\widetilde x,\widetilde y,\epsilon)\vert^2)$. We now drop the tildes. Due to the Hopf bifurcation, we can then subsequently apply a locally defined linear change of coordinates
to put the linear part
\begin{align*}
\begin{pmatrix}
 \partial_x f(X_0,Y_0,\mu,0) & \partial_y f(X_0,Y_0,\mu,0)\\
 \partial_x g(X_0,Y_0,\mu,0) & \partial_y g(X_0,Y_0,\mu,0)
 \end{pmatrix},
\end{align*}
into the normal form
\begin{align*}
\begin{pmatrix}
 \alpha & \beta \\
 -\beta & \alpha
 \end{pmatrix}.
\end{align*}
We denote the new coordinates by $(\widetilde x,\widetilde y)$. This gives
\begin{align*}
 \widetilde x' &= \alpha(\mu)\widetilde x + \beta(\mu)\widetilde y+ \epsilon \widetilde R_0(\mu)+\widetilde R(\widetilde x,\widetilde y,\mu,\epsilon),\\
 \widetilde y' &= -\beta(\mu)\widetilde x + \alpha(\mu)\widetilde y+ \epsilon \widetilde S_0(\mu) + \widetilde S(\widetilde x,\widetilde y,\mu,\epsilon),\\
 \mu' &=\epsilon,
\end{align*}
for real-analytic functions satisfying the same estimates: $\widetilde R(\widetilde x,\widetilde y,\mu,\epsilon)=\kO(\vert (\widetilde x,\widetilde y,\epsilon)\vert^2)$ and
$\widetilde S(\widetilde x,\widetilde y,\mu,\epsilon)=\mathcal O(\vert (\widetilde x,\widetilde y,\epsilon)\vert^2)$. We drop the tildes again and define $( X_1(\mu),Y_1(\mu))$ by
\begin{align} \eqlab{COC}
\begin{pmatrix}
 X_1(\mu)\\
 Y_1(\mu)
\end{pmatrix}
:= -
\begin{pmatrix}
 \alpha(\mu) & \beta(\mu) \\
 -\beta(\mu) & \alpha(\mu)
 \end{pmatrix}^{-1} \begin{pmatrix}
 R_0(\mu)\\
 S_0(\mu) \end{pmatrix},
\end{align}
and introduce the final change of coordinates by
\begin{equation}\eqlab{finalcc}
\begin{aligned}
 x &=\epsilon X_1(\mu) + \epsilon \widetilde x,\\
 y &=\epsilon Y_1(\mu) + \epsilon \widetilde y.
\end{aligned}
\end{equation}
This gives
\begin{align*}
 \widetilde x' &=  \alpha\widetilde x + \beta\widetilde y-\epsilon X_1' + \epsilon^{-1} R(\epsilon (X_1+ \widetilde x),\epsilon (Y_1+\widetilde y),\mu,\epsilon),\\
 \widetilde y' &= -\beta\widetilde x + \alpha\widetilde y-\epsilon Y_1' + \epsilon^{-1} S(\epsilon (X_1+ \widetilde x),\epsilon (Y_1+\widetilde y),\mu,\epsilon),\\
 \mu' &=\epsilon,
\end{align*}
which has the desired form with
\begin{align*}
 \widetilde R(\widetilde x,\widetilde y,\mu,\epsilon) :=  -X_1'+\epsilon^{-2} R(\epsilon (X_1+ \widetilde x),\epsilon (Y_1+\widetilde y),\mu,\epsilon),\\
 \widetilde S(\widetilde x,\widetilde y,\mu,\epsilon) :=-Y_1'+\epsilon^{-2} S(\epsilon (X_1+ \widetilde x),\epsilon (Y_1+\widetilde y),\mu,\epsilon).
\end{align*}

\end{proof}
\begin{remark}\label{RemA1}
   Recall that $\lambda(\mu):=\alpha(\mu)-i\beta(\mu)$ and that by assumption \ref{A1}, we have $\lambda(\mu_*)=0$. This means that the final change of coordinates \eqref{finalcc}, based upon \eqref{COC}, is not well-defined (in general) near $\mu=\mu_*$. An important exception is of course, if $\partial_\epsilon f(X_0,Y_0,\mu)\equiv 0$ and $\partial_\epsilon g(X_0,Y_0,\mu)\equiv 0$. Then $\widetilde R_0=\widetilde S_0=0$ and $X_1=Y_1=0$ and $x=X_0(\mu)$, $y=Y_0(\mu)$, is therefore invariant up to terms of order $\mathcal O(\epsilon^2)$. In this case, we still obtain \eqref{normalform1} regardless of the degeneracy of the matrix in \eqref{COC}. 
   
   In conclusion, our assumptions do not represent the general case of the slow passage through a Hopf.    
   Nevertheless, we see our results as a natural stepping stone, extending the work of \cite{hayes2015a} to a certain class of real-analytic systems. We are confident that our approach can be applied to the setting in \cite{neishtadt2026maximaldelaystabilityloss}, but leave this to future work.

\end{remark}

\bibliography{refs}
\bibliographystyle{plain}
\end{document}

%% file: ex_shared.tex
\usepackage{lipsum}
\usepackage{amsfonts}
\usepackage{graphicx}
\usepackage{epstopdf}
\usepackage{algorithmic}
\ifpdf
  \DeclareGraphicsExtensions{.eps,.pdf,.png,.jpg}
\else
  \DeclareGraphicsExtensions{.eps}
\fi

\newsiamremark{remark}{Remark}
\newsiamremark{hypothesis}{Hypothesis}
\crefname{hypothesis}{Hypothesis}{Hypotheses}
\newsiamthm{claim}{Claim}
\newsiamremark{fact}{Fact}
\crefname{fact}{Fact}{Facts}
\newsiamthm{cor}{Corollary}

\headers{Slow passage through a Hopf in a class of real-analytic system}{J. M. Christensen, K. U. Kristiansen}

\title{On the slow passage through a Hopf in a class of real-analytic systems: Exponential asymptotics and maximal delay\thanks{Submitted to the editors DATE.
\funding{Both authors are supported by Danish Research Council (DFF) Grant 4283-00014B}}}

\author{J. M. Christensen\thanks{Technical University of Denmark, Department of Applied Mathematis and Computer Science, Kgs. Lyngby, DK. 
  (\email{jmoch@dtu.dk})}
\and K. U. Kristiansen\thanks{Technical University of Denmark, Department of Applied Mathematis and Computer Science, Kgs. Lyngby, DK. 
 }}

\usepackage{amsopn}
